\documentclass[11pt,a4paper]{article}
\usepackage{amsmath,amssymb,amsthm,amsfonts,caption,cite,float,graphicx,geometry,parskip,pgfplots,setspace,standalone,subcaption,tikz,titlesec}
\usepackage{color}
\usepackage[nomath]{lmodern}
\usepackage[T1]{fontenc}%
\usepackage[UKenglish]{babel}
\usetikzlibrary{graphs,graphs.standard}%
\graphicspath{{./Images/}}
\allowdisplaybreaks
\newtheorem{thm}{Theorem}[section]

\newtheorem{cor}[thm]{Corollary}

\newtheorem{lem}[thm]{Lemma}

\newtheorem{rem}[thm]{Remark}
\newcommand{\R}{\mathbb{R}}
\newcommand{\N}{\mathbb{N}}
\newcommand*{\basisvector}[2]{\mathbf e_{#1}^{(#2)}}
\newcommand{\ekr}{\basisvector kr}
\newcommand{\elr}{\basisvector lr}
\newcommand{\midthree}{\basisvector 23}
\usepackage[colorlinks]{hyperref}
\usepackage{bookmark}
\newcommand*{\condex}[1]{\mathop{\mathbb E}(#1\mid\mathcal F_n)}
\author{
Haslegrave, John\\
\text{University of Warwick}\\
\text{J.Haslegrave@cantab.net}\\
\and
Jordan, Jonathan\\
\text{University of Sheffield}\\
\text{Jonathan.Jordan@sheffield.ac.uk}\\
\and
Yarrow, Mark\\
\text{University of Sheffield}\\
\text{MAYarrow1@sheffield.ac.uk}\\
}
\title{Condensation in preferential attachment models with location-based choice}
\begin{document}
\maketitle
\begin{abstract}
We introduce a model of a preferential attachment based random graph which extends the family of models in which condensation phenomena can occur. Each vertex has an associated uniform random variable which we call its location. Our model evolves in discrete time by selecting $r$ 
vertices from the graph with replacement, with probabilities proportional to their degrees plus a constant $\alpha$. A new vertex joins the network and attaches to one of these vertices according to a 
given probability associated to the ranking of their locations. We give conditions for the occurrence of condensation, showing the existence of phase transitions in $\alpha$ below which condensation occurs. The condensation in our model differs from that in preferential attachment models with fitness in that the condensation can occur at a random location, that it can be due to a persistent hub, and that there can be more than one point of condensation.

{\bf Keywords}: Preferential Attachment, Fitness, Location, Random Graphs, Phase Transition
\end{abstract}
\section{Introduction}\label{Chapter 3 Section 1}
Preferential attachment graphs were developed as an extension of Erd\H{o}s and R\'enyi's random graph model in order to model evolving networks that exhibited a power law in their degree distributions. The standard preferential attachment method discussed by Barab\'asi and Albert \cite{P.10} evolves from an initial graph $G_0$ with $n_0$ vertices $v_{1-n_0},\ldots,v_0$. For each $n\geq 0$ the graph $G_{n+1}$ is formed by a new vertex $v_{n+1}$ joining $G_n$ and attaching to an existing vertex $V\in\{v_{1-n_0},\dots,v_n\}$ according to
\begin{equation}\label{prefatt}
P(V=v_i)=\frac{\deg_{G_n}(v_i)+\alpha}{\sum_{j=1-n_0}^n(\deg_{G_n}(v_j)+\alpha)},
\end{equation}
for some $\alpha>-1$. Equation \eqref{prefatt} gives the form of preferential attachment developed by Dorogovtsev, Mendes and Samukhin in \cite{P.42} as a generalisation of the Barab\'asi and Albert model found in \cite{P.10}, and we shall use this more general form. However, several of the papers referred to in this section, including \cite{P.10}, only consider the case $\alpha=0$.

It is clear to see from \eqref{prefatt} that vertices with a higher degree have a higher probability of attracting new edges. Some commonly mentioned applications of preferential attachment graphs include the number of links to a website and the growth of the number of connections on social networks.

It is observable in real world networks that the growth in influence of an individual vertex is affected by more factors than just its degree. How attractive the vertex is by itself or in comparison to the others also plays a large part. A model incorporating this notion was introduced by Bianconi and Barab\'{a}si in \cite{P.1}, where they gave each vertex a multiplicative fitness value in their version of \eqref{prefatt}. They did this in order to add an extra dimension to the competition between vertices that joined the vertex using a generalised preferential attachment mechanism. As a consequence, new, fitter, vertices can still compete against vertices which are well-established in the existing network. A particularly interesting feature of preferential attachment with fitness is the so-called condensation phenomena, where at time $n$ a single vertex or a set of vertices of size $o(n)$ (which vertices can depend on time) can have a total degree of order $n$. Condensation for preferential attachment with fitness is studied in detail by Borgs et al.\ \cite{P.2}, Dereich and Ortgiese \cite{P.4} and Dereich, Mailler and M\"{o}rters \cite{P.3}.

Another variant of preferential attachment is the choice model introduced by Malyshkin and Paquette in \cite{P.7,P.8} and Krapivsky and Redner in \cite{P.6}. Here when a new vertex joins the network it first selects several candidate existing vertices at random using \eqref{prefatt}, then attaches to one of the candidates according to a deterministic rule, such as always attaching to the candidate of highest degree. In these papers, depending on the parameters, as the number of vertices increases linear or approximately linear growth can be observed in the degree of the largest vertex. Preferential attachment with degree-based choice was further studied by Haslegrave and Jordan \cite{P.5}, who showed that condensation can also occur when choosing a lower-ranked vertex.

The choice model is combined with fitness in the model studied by Freeman and Jordan \cite{Sou.1} in which each vertex has its own fitness value; the new vertex joins the graph $G_n$ by using preferential attachment to select $r$ vertices from $G_n$. The new vertex forms an edge between itself and the vertex with the highest fitness of the $r$ selections. It is shown in \cite{Sou.1} that again condensation can occur in this model.

In this paper we generalise the model of \cite{Sou.1} to allow for choices other than the largest or smallest, for example selecting the middle vertex of a selection of three. Informally, an example of when the middle of three model might be appropriate is when voting for a political candidate; one might prefer to avoid voting for a candidate is too far left or right, and so decide to cast their vote in the middle. We will also allow for randomised choice rules based on the ranking. Because we are no longer selecting the largest value, we will use the term location rather than fitness. In this paper the locations will be uniform random variables on $[0,1]$; note that as we are only using the order of locations, this is equivalent to any measure without atoms.

Using stochastic approximation techniques, we will show that the normalised empirical measure on the location space given by weighting the location of each vertex of the graph by its degree plus $\alpha$ converges almost surely to a limit. In some cases, this limit is random and has an atom, the emergence of the atom corresponding to condensation occurring in the system. We will show that the atom appears at a random location, as opposed to the results of \cite{P.2, Sou.1} where condensation can only occur at the supremum of the fitness distribution. In addition we will show that condensation in our model can be due to a single vertex which acts as a persistent hub, which is not possible in fitness models. In fact, for some choices of our parameters, condensation has probability 1, but persistent hub behaviour has probability strictly between 0 and 1, so condensation can occur in at least two different ways, each with positive probability.

We will also show that in some cases condensation can occur simultaneously at more than one point; again this is quite different behaviour from that seen in \cite{P.2, Sou.1}. 
Indeed, there are situations where multiple condensation occurs almost surely, and others in which it occurs with probability strictly between $0$ and $1$. 
When using a deterministic choice rule, however, multiple condensation cannot occur and there is always a single phase transition between condensation occurring with probability $0$ and $1$; we give an exact expression for this phase transition for every deterministic rule.

The remainder of this article will start with a discussion of our model and a summary of the main results in Section \ref{Section 2}. Our proofs are in Section \ref{Section 3}, and finally Section \ref{Section 4} includes some specific examples that highlight some of the important aspects of our main results.
\section{Our model and results}\label{Section 2}

\subsection{The model}

Fix a parameter $r\in \N$ with $r\geq 2$, a vector $\Xi\in \R^r$ such that $\Xi_i\in[0,1]$ and $\sum^{r}_{i=1}\Xi_i=1$, and a real number $\alpha>-1$.

We start with a tree $G_{0}$ containing $n_0\geq 2$ vertices which we will denote by $V(G_0)=\bigl\{v_{0},v_{-1},\dots, v_{-(n_0-1)}\bigr\}$. Every vertex $v_i$ in $G_{0}$ has its own location $x_{i}$ in $(0,1)$; we will assume that these locations are distinct. Given $G_n$, at time $n+1$ we form $G_{n+1}$ by adding a new vertex $v_{n+1}$ to the network with a single edge. This vertex has its own uniform random location $x_{n+1}\sim\mbox{Uni}[0,1]$, which is independent of the other $x_i$, and chooses where to attach to at time $n+1$ by selecting a sample of $r$ pre-existing vertices in $G_n$ with replacement with probabilities proportional to their degrees plus $\alpha$ as given by equation \eqref{prefatt}.

Let the $r$ selected vertices at time $n+1$ be denoted by $V_1^{(n+1)},\dots,V_r^{(n+1)}$ with locations $x_{1}^{(n+1)},\allowbreak\dots,x_{r}^{(n+1)}$ respectively, and renumber if necessary so that the locations satisfy $x_{1}^{(n+1)}\leq x_{2}^{(n+1)}\leq\dots\leq x_{r}^{(n+1)}$. For definiteness we specify that if two or more vertices in the selection have the same location, we rank them in the order they were selected; note, however, that with probability $1$ the only way for this to occur is if the same vertex is selected more than once. The probability that $v_{n+1}$ attaches to vertex $V_i^{(n+1)}$ is then given by $\Xi_i$; this choice is independent of the vertices selected and of all choices made at previous time steps. However, this choice need not be independent of $x_i$; while independence between these two variables may seem a natural assumption it is not necessary, and relaxing this assumption may better fit some motivations.

For example if $r=3$ and $\Xi=(0,1,0)$ then the new vertex selects a sample of size $3$, and connects to the selected vertex of median rank. The model can be thought of as generalising the case of the model of \cite{Sou.1} with fixed sample size; that model is obtained by taking our model with $\Xi_r=1$ (or equivalently with $\Xi_1=1$).

Throughout the paper, we shall assume that $G_0$ is a tree. While this requirement is not necessary, and does not change the results, trees are the most natural starting graphs since the attachment process preserves the tree structure. With this assumption, we also have $\sum_{j=1-n_0}^n(\deg_{G_n}(v_j)+\alpha)=2(n+n_0-1)+\alpha(n+n_0)=(n+n_0-1)(2+\alpha)+\alpha$.

\subsection{Results}\label{ss:res}

We define $\Psi_n(x)$ to be the probability that a vertex selected randomly from $G_n$ according to the law \eqref{prefatt} has location less than or equal to $x$, that is \[\Psi_n(x)=\frac{1}{(n+n_0-1)(2+\alpha)+\alpha}\sum_{v_i\in V(G_n):x_i\leq x}(\deg_{G_n}(v_i)+\alpha).\] Clearly $\Psi_n(0)=0$, almost surely, and $\Psi_n(1)=1$. We can think of $\Psi_n$ as being the distribution function of the normalised empirical measure on the location space given by weighting the location of each vertex of the graph by its degree plus $\alpha$; we will label this measure $\nu_n$.

Our first result is on the convergence of the measures $\nu_n$. For a fixed $x$ and choice of $\Xi$, define \[F_1(y;x,\Xi)=x(\alpha+1)-(2+\alpha)y+\sum^{r}_{l=1}\Xi_l \sum^{r}_{i=l}\binom{r}{i}y^{i}(1-y)^{r-i},\] to be considered as a function of $y$ for $y\in[0,1]$.

We will say that $p\in (0,1)$ is a \emph{stable zero} of $F_1(y;x,\Xi)$ if $F_1(p;x,\Xi)=0$ where there exists an $\epsilon$ such that for $y\in (p-\epsilon,p)$ we have $F_1(y;x,\Xi)>0$ and for $y\in (p,p+\epsilon)$ we have $F_1(y;x,\Xi)<0$. Similarly $p\in (0,1)$ is an \emph{unstable zero} if $F_1(p;x,\Xi)=0$ and there exists $\epsilon$ such that for $y\in (p-\epsilon,p)$ we have $F_1(y;x,\Xi)<0$ and for $y\in (p,p+\epsilon)$ we have $F_1(y;x,\Xi)>0$, and $p\in (0,1)$ is a \emph{touchpoint} if $F_1(p;x,\Xi)=0$ and there exists $\epsilon$ such that we have either $F_1(y;x,\Xi)<0$ for all $y\in (p-\epsilon,p+\epsilon)\setminus\{p\}$ or $F_1(y;x,\Xi)>0$ for all $y\in (p-\epsilon,p+\epsilon)\setminus\{p\}$.

\begin{rem}Since $F_1(y;x,\Xi)$ is a polynomial, every root in $(0,1)$ is either a stable zero, an unstable zero or a touchpoint. Also, for $x\in(0,1)$ we have $F_1(0;x,\Xi)>0>F_1(1;x,\Xi)$, so $0$ and $1$ are not roots.\end{rem}

\begin{thm}\label{main1}
As $n\to\infty$, the sequence of measures converges weakly, almost surely, to a (possibly random) probability measure on $[0,1]$, whose distribution function we will call $\Psi$. Furthermore, the following properties hold for any given $x\in (0,1)$. \begin{enumerate}
\item Almost surely, $\Psi(x)$ is a zero of the function $F_1(y;x,\Xi)$.
\item For any stable zero or touchpoint $y$ of $F_1(y;x,\Xi)$, there is positive probability that $\Psi(x)=y$.
\item Any unstable zero $y$ of $F_1(y;x,\Xi)$ has probability zero that $\Psi(x)=y$.
\end{enumerate}
\end{thm}

Depending on the parameters of the model, the limit $\Psi$ may be continuous or discontinuous; for example in Section \ref{midofthreemodd} we will show that the model mentioned above where $r=3$ and $\Xi=(0,1,0)$ exhibits a phase transition where $\Psi$ is almost surely continuous for $\alpha\geq -\frac{1}{2}$ and almost surely has a discontinuity for $\alpha<-\frac{1}{2}$.

Discontinuity of $\Psi$ implies that $\Psi_n$ increases by $\Theta(1)$ on an interval
of length $o(1)$ as $n\to\infty$; this corresponds to a condensation phenomenon whereby a small number of vertices with locations in a range of size $o(1)$ have a
$\Theta(1)$ probability of being selected. A consequence of Theorem \ref{main1} is that where there is an interval of $x$ values for which $F_1(y;x,\Xi)$ has more than one stable root, the discontinuity occurs at a random location, as any stable root has positive probability of being the limit for each $x$ in the interval.

It does not immediately follow from discontinuity of $\Psi$ that a single vertex has linear degree; however, the next result shows that this occurs with positive probability. Without loss of generality, we will focus on the vertex $v_0$, present in the graph from the start, and label its location as $z$. We define
\begin{equation}
D_n=\frac{\alpha+\deg_{G_n}(v_0)}{(n+n_0-1)(2+\alpha)+\alpha},\label{def-dn}
\end{equation}
which would be the probability of selecting $v_0$ for attachment under the preferential attachment rule.

\begin{thm}\label{main2}Let $z$ be the location of vertex $v_0$.
If $y_i\geq y_j$ are two stable fixed points of $F_1(y;z,\Xi)$, then there is positive probability that $(\Psi_n(z),D_n)\to (y_i,y_i-y_j)$ as $n\to\infty$.
\end{thm}

Theorem \ref{main2} shows that if there are two distinct stable fixed points of $F_1(y;z,\Xi)$ then condensation can occur at a persistent hub in the sense that a vertex of the initial graph with location $z$ has its degree growing linearly with $n$ with positive probability. The condensation phenomenon that occurs in this case is thus different from that found for preferential attachment with multiplicative fitness, where Dereich, Mailler and M\"{o}rters \cite{P.3} show that the maximum degree divided by $n$ converges to zero in probability; it is also distinct from that found by Freeman and Jordan \cite{Sou.1} where, although the maximum degree is usually of linear order in $n$, any individual vertex only dominates for a finite time before being displaced by fitter vertices.

However, the next result shows that for some choices of our parameters there is also positive probability that the condensation phenomenon is not due to a persistent hub, as it implies there is positive probability of the condensation occurring at a specific location, where the probability of there being a vertex is zero. This suggests that in this case the condensation phenomenon is more like one of those found in \cite{P.3} or \cite{Sou.1}, in that vertices whose location is close to the condensation location are replaced over time in the condensate by those which are even closer.

\begin{thm}\label{main3}
Let $x\in(0,1)$ and $\Xi$ be such that there exists $p\in (0,1)$ which is a touchpoint of $F_1(y;x,\Xi)$. Then there is positive probability that condensation occurs at $p$ in the sense that $\Psi$ has a discontinuity at $p$.
\end{thm}

One natural question is whether it is possible to have more than one discontinuity in $\Psi$, implying more than one point of condensation. 
The following result shows that this is not possible in the case where the same rank is always chosen.

\begin{thm}\label{highlander}
Whenever $\Xi_k=1$ for some $k\in\{1,2,\dots,r\}$, it is impossible to have more than one point of condensation.
\end{thm}
An example of a choice of $\Xi$ for which more than one point of condensation is possible appears in section \ref{secondorsixth}.

In the case where the same rank is always chosen, we can give a precise description of the phase transition.
\begin{thm}\label{threshold}
If $\Xi_k=1$ then condensation almost surely occurs if $\alpha<\alpha_{\mathrm c}$ and almost surely does not occur if $\alpha\geq\alpha_{\mathrm c}$, where
$\alpha_{\mathrm c}=r-2$ if $k=1$ or $k=r$, and
\begin{equation}
\alpha_{\mathrm c}=r\binom{r-1}{k-1}\frac{(k-1)^{k-1}(r-k)^{r-k}}{(r-1)^{r-1}}-2\label{alpha-crit}
\end{equation}
otherwise.\end{thm}
In particular, $\alpha_{\mathrm c}>-1$ for all values of $r$ and $k$, meaning that there is always a genuine phase transition. 
In fact $\alpha_{\mathrm c}>0$, giving condensation even for linear preferential attachment, except in a few small cases: 
first (or second) of two choices; middle of three; second (or third) of four; and middle of five.

\section{Proofs}\label{Section 3}

For the majority of this section we will restrict the model to the case where the choice between the $r$ selected vertices is deterministic, i.e.\ $v_{n+1}$ always attaches to the selected vertex with the $k$th lowest location for some fixed $k$. In the formal notation given above this model can be written as $\Xi=(0,0,\dots,1,\dots,0,0)$ where the $1$ is in the $k$th position; write $\ekr$ for this vector. We will deal with general $\Xi$ in Section \ref{gen-case}.

A key technique we use in our proofs is that of of stochastic approximation algorithms, originally developed by Robbins and Monro \cite{P.28}. 
Stochastic approximation methods appear naturally in preferential attachment models, and have been used, for example, 
by Malyshkin and Paquette \cite{P.8} and Dereich and Ortgiese \cite{P.4}. 
Stochastic approximation processes operate in discrete time with standard notation, based on Pemantle \cite{P.9},
\begin{equation}\label{gen-sae}
X_{n+1}-X_{n}=\gamma_n(F(X_n)+\xi_{n+1}+R_{n}),
\end{equation}
where $\{X_n,n\geq 1\}$ is a sequence of random variables on $\mathbb{R}^d$, $\gamma_n$ are step sizes satisfying $\sum_{n=1}^{\infty}\gamma_n=\infty$ 
and $\sum_{n=1}^{\infty}\gamma_n^2<\infty$, $F$ is a function from $\mathbb{R}^d$ to itself, $R_n$ are remainder terms which must tend to zero and satisfy $\sum_{n=1}^{\infty}n^{-1}|R_n|<\infty$, and $\xi_{n+1}$ are noise terms satisfying $\condex{\xi_{n+1}}=0$. 

We will mainly use results found in Section 2 of \cite{P.9}, which show that under certain conditions the process will converge almost surely to an equilibrium of $F$, that stable equilibria have positive probability of being the limit and that unstable equilibria usually do not.

\subsection{Proofs of Theorems \ref{main1} to \ref{main3} for deterministic choice rules}

Let $\mathcal{F}_{n}$ be the $\sigma$-algebra generated by the graphs $G_n$ and the locations of their vertices up until time $n$, i.e.\ $\mathcal{F}_{n}=\sigma(G_i,x_{i};i\leq n)$.
For $x\in [0,1]$, the probability of attaching to a vertex with location in $[0,x]$ at time $n+1$, conditional on $\mathcal F_n$, is given by
\begin{equation}\label{prob select vert leq x}
g\Bigl(\Psi_n(x);\ekr\Bigr)= \sum^{r}_{i=k}\binom ri\Psi_n(x)^{i}(1-\Psi_n(x))^{r-i}.
\end{equation}

We can now formulate the first stochastic approximation equation associated to our model, which will allow us to show that as the network grows the 
total weight of vertices with location in $[0,x]$ grows linearly. If $G_0$ has $n_0$ vertices and $e_0$ edges then $G_n$ has $n_0+n$ vertices and $e_0+n$ edges,
giving the normalising constant $\gamma_n=(2(e_0+n)+\alpha(n_0+n))^{-1}$. Assuming $G_0$ is a tree, we have the simpler form
$\gamma_n=((n+n_0-1)(2+\alpha)+\alpha)^{-1}$.

\begin{lem}\label{first sae} For a fixed $x\in[0,1]$, we have the stochastic approximation equation
\begin{equation}\label{f1(y;x,ekr)}
\Psi_{n+1}(x)-\Psi_{n}(x)=\gamma_{n+1}\Bigl(F_1\Bigl(\Psi_{n}(x);x,\ekr\Bigr)+\xi_{n+1}\Bigr),
\end{equation}
where $F_1\Bigl(y;x,\ekr\Bigr)=g\Bigl(y;\ekr\Bigr)-(2+\alpha)y+x(1+\alpha)$ and $\condex{\xi_{n+1}}=0$.
\end{lem}
\begin{proof}
Equation \eqref{f1(y;x,ekr)} is in the form of \eqref{gen-sae} with $F$ mapping $y$ to $F_1\Bigl(y;x,\ekr\Bigr)$ and $R_n\equiv 0$; 
clearly $\gamma_n$ has the required properties and so it suffices to show that a suitable $\xi_{n+1}$ may be defined.

Note that $\gamma_n^{-1}\Psi_n(x)$ is the total weight at time $n$ of all vertices located in $[0,x]$, where each vertex $v$ has weight $\deg_{G_n}(v)+\alpha$. 
Thus we consider the change to this total weight from the new vertex and edge; say that $v_{n+1}$ attaches to $w_{n+1}$.

The vertex $w_{n+1}$ has location at most $x$ with probability $g\Bigl(\Psi_{n}(x);\ekr\Bigr)$ from \eqref{prob select vert leq x} 
as the probability of attaching to the vertex with rank $k$ of $r$ selections. The expected addition to the total weight arising from 
the location of $v_{n+1}$ is $x(1+\alpha)$. Therefore
\begin{align*}
\condex{\gamma_{n+1}^{-1}\Psi_{n+1}(x)}&=\gamma_n^{-1}\Psi_{n}(x)+x(1+\alpha)+g\Bigl(\Psi_{n}(x);\ekr\Bigr)\\
&=\gamma_{n+1}^{-1}\Psi_n(x)-(2+\alpha)\Psi_{n}(x)+x(1+\alpha)+g\Bigl(\Psi_{n}(x);\ekr\Bigr)\\
&=\gamma_{n+1}^{-1}\Psi_n(x)+F_1\Bigl(\Psi_{n}(x);x,\ekr\Bigr),
\end{align*}
and so
\[\condex{\Psi_{n+1}(x)}=\Psi_{n}(x)+\gamma_{n+1}F_1\Bigl(\Psi_{n}(x);x,\ekr\Bigr).\]
Defining $\xi_{n+1}$ as
\begin{equation*}
\xi_{n+1}=\frac{\Psi_{n+1}(x)-\condex{\Psi_{n+1}(x)}}{\gamma_{n+1}},
\end{equation*}
we have \eqref{f1(y;x,ekr)} and clearly $\condex{\xi_{n+1}}=0$.
\end{proof}
Before continuing, we summarise the necessary results on one-dimensional stochastic approximations; 
recall the definitions of stable zeros, unstable zeros and touchpoints from Section \ref{ss:res}.
\cite[Corollary 2.7]{P.9} and \cite[Theorem 2.8]{P.9} apply on the assumption that $F$ is bounded and continuous and $\condex{\xi_{n+1}^2}\leq K$ for some finite $K$, 
and say, respectively, that $X_n$ converges almost surely to a zero of $F$ and that any stable zero has positive probability of being the limit. 
\cite[Theorem 2.9]{P.9} requires additionally that $\condex{\xi_{n+1}^+}$ and $\condex{\xi_{n+1}^-}$ are bounded above and below by positive numbers 
in a neighbourhood of $p$, and says that if $p$ is an unstable zero then it is almost surely not the limit. 
Finally, Theorem 2.5 of Antunovi\'{c}, Mossel and R\'{a}cz \cite{P.43}, based on work by Pemantle in \cite{P.44}, requires that $F$ is continuously differentiable 
and $\lvert\xi_{n+1}\rvert\leq K$ almost surely, and says that each touchpoint has positive probability of being the limit.
\begin{thm}\label{equstabroot}
Let $x\in(0,1)$. The sequence of random variables $\Psi_n(x)$ converges almost surely to a zero of $F_1\Bigl(y;x,\ekr\Bigr)$. Any stable zero in $[0,1]$ or touchpoint in $(0,1)$ has positive probability of being the limit, while any unstable zero has probability zero of being the limit.
\end{thm}
\begin{proof}
First we note that $F_1\Bigr(0;x,\ekr\Bigr)>0$ and $F_1\Bigr(1;x,\ekr\Bigr)<0$. Therefore there must be at least one zero of $F_1\Bigl(y;x,\ekr\Bigr)$ in the interval $[0,1]$.

Since $y\mapsto F_1\Bigl(y;x,\ekr\Bigr)$ is a polynomial, it is continuously differentiable. Consequently, in order to apply the results detailed above, it suffices to check that $\xi_{n+1}$ has the required properties.

Recall that $\gamma_{n+1}^{-1}\Psi_{n+1}(x)$ is the total weight of vertices in $[0,x]$ at time $n+1$, so $\xi_{n+1}$ is the difference between the actual value of this quantity and its expectation at time $n$. Recall that $v_{n+1}$ has location $x_{n+1}$, and write $z_{n+1}$ for the location of its neighbour in $G_{n+1}$. 
Given $\mathcal F_n$, the total weight only depends on $x_{n+1}$ and $z_{n+1}$,
and takes values in $[\gamma_n^{-1}\Psi_n(x),\gamma_n^{-1}\Psi_n(x)+2+\alpha]$. 
Thus $\gamma_{n+1}^{-1}\condex{\Psi_{n+1}(x)}$ is also in this interval, meaning that
$\lvert\xi_{n+1}\rvert\leq 2+\alpha$, and so $\condex{\xi_{n+1}}\leq(2+\alpha)^2$. 

We now verify the conditions on $\xi_{n+1}^{+}$ and $\xi_{n+1}^{-}$. Since $\lvert\xi_{n}\rvert=\xi_{n}^{+}+\xi_{n}^{-}$, both of these are also bounded above.
Fix an unstable zero $p\in(0,1)$ and let $\epsilon$ be such that $0<p-\epsilon<p+\epsilon<1$. Provided $\Psi_n(x)\in(p-\epsilon,p+\epsilon)$, we have 
$\mathbb{P}(x_{n+1},z_{n+1}\leq x\mid\mathcal F_n)=xg\Bigl(\Psi_n(x);\ekr\Bigr)$ is bounded away from $0$, and similarly for $\mathbb{P}(x_{n+1},z_{n+1}> x\mid\mathcal F_n)$.
It follows that $\condex{\gamma_{n+1}^{-1}\Psi_{n+1}(x)}$ is bounded away from $\gamma_n^{-1}\Psi_n(x)$ but attains this value with probability bounded away from $0$, giving a lower bound on $\condex{\xi_{n+1}^-}$; similar reasoning applies to $\xi_{n+1}^+$.

Thus all the results described above apply in this setting, giving almost sure convergence to the zero set, positive probability of convergence to each stable zero or touchpoint,
and zero probability of convergence to each unstable zero.\end{proof}

The following result completes the proof of Theorem \ref{main1} in the case $\Xi=\ekr$.
\begin{cor}\label{limpsi}The sequence of measures defined by $\Psi_n$ converges weakly, almost surely, to a limit defined by a (possibly random) distribution function $\Psi:\mathbb R\to[0,1]$.
\end{cor}
\begin{proof}By definition, we have that for each $n$ $\Psi_n$ is a non-decreasing cadlag function with $\Psi_n(1)=1$ and, almost surely, $\Psi_n(0)=0$. We apply Theorem \ref{equstabroot} to a countable dense set of $x\in(0,1)$ and for $x$ in this set we define $\hat{\Psi}(x)=\lim_{n\to\infty}\Psi_n(x)$. We then define a cadlag function $\Psi$ by defining $\Psi(x)=\inf_{y>x}\hat{\Psi}(y)$ for $x\in[0,1)$, $\Psi(x)=0$ for $x<0$ and $\Psi(x)=1$ for $x\geq 1$.  By this construction, the probability measure with distribution function $\Psi$ is a weak limit of the sequence of probability measures with distribution functions $\Psi_n$.\end{proof}


We now move towards proving Theorem \ref{main2} in the case $\Xi=\ekr$. To do this, we consider a two dimensional stochastic approximation for $(\Psi_n(z),D_n)$, where $D_n$ is given by \eqref{def-dn}. Let $\chi_n$ be the location of a selected vertex under preferential attachment from $G_n$. Assuming that $v_0$ is the only vertex at location $z$, which occurs almost surely, we have
\begin{align*}
P(\chi_n =z\mid\mathcal{F}_{n})=&D_{n},\\
P(\chi_n <z\mid\mathcal{F}_{n})=&\Psi_{n}(z)-D_{n},\\
P(\chi_n >z\mid\mathcal{F}_{n})=&1-\Psi_{n}(z).
\end{align*}
The probability of the $k$th ranked location being $z$, and hence of selecting vertex $v_0$ for $v_{n+1}$ to attach to is given by
\begin{equation}\label{equ of select vert v1}
h\Bigl(\Psi_n(z),D_n;\ekr\Bigr)=\sum^{k-1}_{j=0}\sum^{r}_{i=k}\binom{r}{i}\binom{i}{j}(\Psi_n(z)-D_n)^jD_n^{i-j}(1-\Psi_n(z))^{r-i}.
\end{equation}
We can now form our two dimensional stochastic approximation.
\begin{lem}\label{second sae}
We have the stochastic approximation equation
\begin{equation}
D_{n+1}-D_{n}=\gamma_{n+1}\Bigl(F_{2}\Bigl(\Psi_{n}(z),D_{n};\ekr\Bigr)+\zeta_{n+1}\Bigr),\label{sec-sae-eq}
\end{equation}
where $F_{2}\Bigl(y,d;\ekr\Bigr)=h\Bigl(y,d;\ekr\Bigr)-(2+\alpha)d$ and $\mathop{\mathbb{E}}(\zeta_{n+1}\mid\mathcal{F}_{n})=0$.
\end{lem}
\begin{proof}
As in the proof of Lemma \ref{first sae}, \eqref{sec-sae-eq} is a stochastic approximation with $R_n\equiv 0$ provided that a suitable $\zeta_{n+1}$ may be defined. Writing $w_{n+1}$ for the vertex $v_{n+1}$ attaches to, we have
\begin{align*}\condex{D_{n+1}}&=\gamma_{n+1}\bigl(\mathop{\mathbb{E}}\bigl(\deg_{G_{n+1}}(v_0)\bigr)+\alpha\bigr)\\
&=\gamma_{n+1}\bigl(\deg_{G_{n}}(v_0)+\alpha+\mathbb{P}(w_{n+1}=v_0\mid\mathcal F_n)\bigr)\\
&=\gamma_{n+1}\Bigl(\gamma_n^{-1}D_n+h\Bigl(\Psi_n(z),D_n;\ekr\Bigr)\Bigr),\end{align*}
using \eqref{equ of select vert v1}. Since $\gamma_n^{-1}\gamma_{n+1}^{}=1-(2+\alpha)\gamma^{}_{n+1}$, 
it follows that
\[\condex{D_{n+1}}=D_n+\gamma_{n+1}F_{2}\Bigl(\Psi_{n}(z),D_{n};\ekr\Bigr).\]
Defining
\begin{equation*}
\zeta_{n+1}=\frac{D_{n+1}-\mathop{\mathbb{E}}(D_{n+1}\mid\mathcal{F}_{n})}{\gamma_{n+1}},
\end{equation*}
we have \eqref{sec-sae-eq} with $\condex{\zeta_{n+1}}=0$.
\end{proof}
We have now formed a two dimensional system of stochastic approximation equations represented by
\begin{equation}
\begin{pmatrix}
\Psi_{n+1}(z) \\
D_{n+1} 
\end{pmatrix}
-
\begin{pmatrix}
\Psi_{n}(z) \\
D_{n} 
\end{pmatrix}
=\gamma_{n+1}
\begin{pmatrix}
F_{1}(\Psi_{n}(z);x,\Xi) \\
F_{2}(\Psi_{n}(z),D_{n};\Xi) 
\end{pmatrix}
+
\gamma_{n+1}
\begin{pmatrix}
\xi_{n+1} \\
\zeta_{n+1}
\end{pmatrix}.
\end{equation}

The following relationship between $F_1$ and $F_2$ will be useful for identifying stationary points of the vector field associated to our two dimensional stochastic approximation.
\begin{thm}\label{relationship theorem}
We have that
\begin{equation}\label{relationship equation}
F_1\Bigl(y-d;x,\ekr\Bigr)=F_1\Bigl(y;x,\ekr\Bigr)-F_2\Bigl(y,d;\ekr\Bigr).
\end{equation}
\end{thm}
\begin{proof}
We use induction on $k$. For $k=1$,
\begin{align*}
F_1\Bigl(y-d;x,\basisvector{1}{r}\Bigr)={}&1-(1-y+d)^r -(2+\alpha)(y-d)+x(1+\alpha)\\
={}&-(2+\alpha)(y) + x(1+\alpha) + \sum^{r}_{i=1}\binom{r}{i}y^{i}(1-y)^{r-i} \\
&-\Biggl(-(2+\alpha)d+\sum^{r}_{i=1}\binom{r}{i}d^{i}(1-y)^{r-i}\Biggr)\\
={}&F_1\Bigl(y;x,\basisvector{1}{r}\Bigr)-F_2\Bigl(y,d;\basisvector{1}{r}\Bigr)
\end{align*}
Assuming \eqref{relationship equation} holds for $k$,
\begin{align*}
F_1\Bigl(y-d;x,\basisvector{k+1}{r}\Bigr)
={}&\sum^{r}_{i=k+1}\binom{r}{i}(y-d)^{i}(1-y+d)^{r-i}-(2+\alpha)(y-d)+x(\alpha+1)\\
={}&F_1\Bigl(y-d;x,\ekr\Bigr)-\binom{r}{k}(y-d)^{k}(1-y+d)^{r-k}\\
={}&F_1\Bigl(y;x,\ekr\Bigr)-F_2\Bigl(y,d;\ekr\Bigr)-\binom{r}{k}(y-d)^{k}(1-y+d)^{r-k}\\ ={}&F_1\Bigl(y;x,\basisvector{k+1}{r}\Bigr)-F_2\Bigl(y,d;\ekr\Bigr)-\binom{r}{k}(y-d)^k(1-y+d)^{r-k}\\ &+\binom{r}{k}y^k(1-y)^{r-k}\end{align*} where in the last line we have separated out the first term in the sum in the definition of $F_1\Bigl(y;x,\ekr\Bigr)$.
By the definition of $F_2\Bigl(y,d;\ekr\Bigr)$ and binomial expansion of $(1-y+d)^{r-k}$ and $y^k=(y-d+d)^k$, we get that $F_1\Bigl(y-d;x,\basisvector{k+1}{r}\Bigr)$ is equal to \begin{align*}
{}&F_1\Bigl(y;x,\basisvector{k+1}{r}\Bigr)-(2+\alpha)d-\sum^{k-1}_{j=0}\sum^{r}_{i=k}\binom{r}{i}\binom{i}{j}(y-d)^jd^{i-j}(1-y)^{r-i}\\
&-\sum^{r}_{i=k}\binom{r}{i}\binom{i}{k}(y-d)^kd^{i-k}(1-y)^{r-i}+\sum^{k}_{j=0}\binom{r}{k}\binom{k}{j}(y-d)^jd^{k-j}(1-y)^{r-k}.\end{align*}
Re-arranging the sums gives that 
\begin{align*}
{}F_1\Bigl(y-d;x,\basisvector{k+1}{r}\Bigr)=&F_1\Bigl(y;x,\basisvector{k+1}{r}\Bigr)-\sum^{k-1}_{j=0}\sum^{r}_{i=k}\binom{r}{i}\binom{i}{j}(y-d)^jd^{i-j}(1-y)^{r-i}\\
&-\sum^{k}_{j=k}\sum^{r}_{i=k}\binom{r}{i}\binom{i}{j}(y-d)^jd^{i-j}(1-y)^{r-i} \\
&+\sum^{k}_{j=0}\sum^{k}_{i=k}\binom{r}{i}\binom{i}{j}(y-d)^jd^{i-j}(1-y)^{r-i}-(2+\alpha)d\\
={}&F_1\Bigl(y;x,\basisvector{k+1}{r}\Bigr)-\Biggl(\sum^{k}_{j=0}\sum^{r}_{i=k+1}\binom{r}{i}\binom{i}{j}(y-d)^jd^{i-j}(1-y)^{r-i}+(2+\alpha)d\Biggr)\\
={}&F_1\Bigl(y;x,\basisvector{k+1}{r}\Bigr)-F_2\Bigl(y,d;\basisvector{k+1}{r}\Bigr),
\end{align*}
completing the proof.\end{proof}

Let $\mathcal{Y}(z)$ be the set $\Bigl\{y:F_1\Bigl(y_i;z,\ekr\Bigr)=0\Bigr\}$, and write its elements $y_1,y_2,\ldots,y_{|\mathcal{Y}(z)|}$.  It then follows from Theorem \ref{relationship theorem} that $F_2\Bigl(y_i,y_i-y_j;z,\ekr\Bigr)=0$ and that the solutions to $F_1\Bigl(y;z,\ekr\Bigr)=F_2\Bigl(y,d;\ekr\Bigr)=0$ all take the form $(y,d)=(y_i,y_i-y_j)$ where $i,j\in\{1,2,\dots,|\mathcal{Y}(z)|\}$.  Note also that if $y<\frac{(1+\alpha)x}{2+\alpha}$ then $F_1\Bigl(y_i;z,\ekr\Bigr)>0$, and similarly if $y>1-\frac{(1+\alpha)(1-x)}{2+\alpha}$ then $F_1\Bigl(y_i;z,\ekr\Bigr)<0$, so any $y\in \mathcal{Y}(z)$ satisfies \begin{equation}\label{y-inequality}\frac{(1+\alpha)x}{2+\alpha}\leq y \leq 1-\frac{(1+\alpha)(1-x)}{2+\alpha}.\end{equation}

To investigate the stability of the stationary points, we will now calculate the Jacobian $M$ of the two dimensional system. We can observe that $M$ is an upper triangular matrix because $F_1\Bigl(y;z,\ekr\Bigr)$ does not depend on $d$ so $\frac{\partial F_1}{\partial d}=0$. Therefore the eigenvalues of our system are
\begin{align*}
\lambda_1\Bigl(y;\ekr\Bigr)&=\sum^{r}_{i=k}\binom{r}{i}y^{i-1}(1-y)^{r-i-1}(i-ry)-(2+\alpha),\\
\lambda_2\Bigl(y,d;\ekr\Bigr)&=\sum^{k-1}_{j=0}\sum^{r}_{i=k}\binom{r}{i}\binom{i}{j}Q_{i,j}^{(r)}(y,d)(y-d)^{j-1}d^{i-j-1}(1-y)^{r-i-1}
-(2+\alpha),
\end{align*}
where $Q_{i,j}^{(r)}(y,d)=(y-d)(i-iy+rd-id)+j(1-y)(2d-y)$.

\begin{thm}\label{only one eig needed}
For any $y_i,y_j\in\mathcal{Y}(z)$ such that $y_i-y_j\geq0$ and
\begin{equation*}
\frac{\partial}{\partial y}F_1\Bigl(y;z,\ekr\Bigr)=\lambda_1\Bigl(y;\ekr\Bigr)<0
\end{equation*}
is satisfied for both $y=y_i$ and $y=y_j$,
we have that $(y_i,y_i-y_j)$ is a stable equilibrium of the vector field $\Bigl(F_1\Bigl(y;z,\ekr\Bigr),F_2\Bigl(y,d;\ekr\Bigr)\Bigr)$.
\end{thm}

\begin{proof} By rearranging \eqref{relationship equation} we can see that
\[
F_2\Bigl(y,d;\ekr\Bigr)=F_1\Bigl(y;z,\ekr\Bigr)-F_1\Bigl(y-d;z,\ekr\Bigr)
\]
and can deduce that
\[\lambda_2\Bigl(y,d;\ekr\Bigr)=\frac{\partial}{\partial d}\Bigl(F_1\Bigl(y;z,\ekr\Bigr)-F_1\Bigl(y-d;z,\ekr\Bigr)\Bigr).\]
Here $F_1\Bigl(y;z,\ekr\Bigr)$ does not depend on $d$, giving $\frac{\partial}{\partial d}F_1\Bigl(y;z,\ekr\Bigr)=0$. Consequently, $\lambda_2\Bigl(y,d;\ekr\Bigr)=-\frac{\partial}{\partial d}F_1\Bigl(y-d;z,\ekr\Bigr)=\lambda_1\Bigl(y-d;\ekr\Bigr)$. All roots of $F_1\Bigl(y;z,\ekr\Bigr)=F_2\Bigl(y,d;\ekr\Bigr)=0$ are of the form $(y_i,y_i-y_j)$. Evaluating eigenvalues at this point gives $\lambda_1\Bigl(y_i;\ekr\Bigr)$ and $\lambda_2\Bigl(y_i,d_i;\ekr\Bigr)=\lambda_1\Bigl(y_j;\ekr\Bigr)$, which, referring to our initial conditions, are both negative. Therefore the pair $y_i$ and $y_j$ form the possible limit $(y_i,y_i-y_j)$.
\end{proof}

\begin{cor}\label{2dlimits}If $y_i\geq y_j$ are two stable fixed points of $F_1\Bigl(y;z,\ekr\Bigr)$, then there is positive probability of $(\Psi_n(z),D_n)\to (y_i,y_i-y_j)$ as $n\to\infty$.\end{cor}
\begin{proof} Theorem \ref{only one eig needed} shows that $(y_i,y_i-y_j)$ is a stable stationary point of the vector field $\Bigl(F_1\Bigl(y;z,\ekr\Bigr),F_2\Bigl(y,d;\ekr\Bigr)\Bigr)$. The conclusion then follows from Theorem 2.16 of Pemantle \cite{P.9}, as long as we can show that there is no $t$ for which $\bigl(\Psi_{t+n}(z),D_{t+n}\bigr)_{n\geq 0}$ almost surely avoids some neighbourhood of $(y_i,y_i-y_j)$.  To see this, note that $D_{t+n}$ depends only on the number of vertices added up to time $t+n$ which connect to the vertex $v_0$, and that $\Psi_{t+n}(z)$ depends only on the number of vertices added up to time $t+n$ which connect to vertices with location less than or equal to $z$ and the number of those vertices which have locations less than or equal to $z$ themselves.

Any integer between $\deg_{G_0}(v_0)$ and $\deg_{G_0}(v_0)+t+n$ inclusive has positive probability as a value with for $\deg_{G_{t+n}}(v_0)$; by the definition of $D_{t+n}$ this ensures that for sufficiently large $n$ there is positive probability of $D_{t+n}$ being in any given subinterval of $\bigl[0,\frac{1}{2+\alpha}\bigr]$.  By similar arguments for $\Psi_{t+n}(z)$, $\bigl(\Psi_{t+n}(z),D_{t+n}\bigr)_{n \geq 0}$ can, with positive probability, approach $(y_i,y_i-y_j)$ arbitrarily closely, as long as $0\leq y_i\leq 1$ and $0\leq y_i-y_j \leq \frac{1}{2+\alpha}$.  That these conditions are satisfied follows from \eqref{y-inequality} and the assumption that $y_i\geq y_j$.\end{proof}

This completes the proof of Theorem \ref{main2} in the case $\Xi=\ekr$.

To prove Theorem \ref{main3}, we first note that where $p$ is a touchpoint of $F_1\Bigl(y;x,\ekr\Bigr)$ with $F_1\Bigl(y;x,\ekr\Bigr)$ non-positive in a neighbourhood of $p$ there will be a neighbourhood of $p$ which contains no zeros of $F_1(y;x-u,\Xi)$ for positive $u$. Hence the probability of $\Psi(x-u)$ being in this neighbourhood of $p$ is zero, but we know from Theorem \ref{equstabroot} that there is positive probability that $\lim_{n\to\infty} \Psi_n(x)=p$. Hence there is positive probability of a discontinuity at $x$. The same applies, with $x-u$ replaced by $x+u$, if $F_1\Bigl(y;x,\ekr\Bigr)$ is non-negative in a neighbourhood of $p$, completing the proof.

\subsection{Proofs of Theorems \ref{main1} to \ref{main3} in the general case}\label{gen-case}

We now extend the proofs of Theorems \ref{main1} to \ref{main3} in the case where $\Xi$ is not necessarily equal to $\ekr$. We can derive
\begin{align*}
F_1(y;x,\Xi)&=\sum^r_{l=1}\Xi_l F_1\Bigl(y;x,\elr\Bigr)\\
&=x(\alpha+1)-(2+\alpha)y+\sum^{r}_{l=1}\Xi_l \sum^{r}_{i=l}\binom{r}{i}y^{i}(1-y)^{r-i}
\end{align*}
and extend the definition of $F_2$ from Lemma \ref{second sae} as
\begin{align*}
F_2(y,d;\Xi)&=\sum^r_{l=1}\Xi_l F_2\Bigl(y,d;\elr\Bigr)\\
&=\Biggl(\sum^{r}_{l=1}\Xi_l \sum^{l-1}_{j=0}\sum^{r}_{i=l}\binom{r}{i}\binom{i}{j}(y-d)^j d^{i-j}(1-y)^{r-i}\Biggr)-(2+\alpha)d.
\end{align*}
We can see that Lemmas \ref{first sae} and \ref{second sae} still hold, and the arguments for Theorem \ref{equstabroot} and Corollary \ref{limpsi} work in the same way as for the case $\Xi=\ekr$, completing the proof of Theorem \ref{main1}.

By considering the above expressions for $F_1(y;x,\Xi)$ and $F_2(y;x,\Xi)$ as weighted sums of $F_1\Bigl(y;x,\elr\Bigr)$ and $F_2\Bigl(y;x,\elr\Bigr)$ respectively, we can see that Theorem \ref{relationship theorem} still holds, so if we let $\mathcal{Y}(z)=\{y_1,y_2,\dots,y_{\lvert\mathcal{Y}(z)\rvert}\}$ be the set of zeros of $F_1(y;z,\Xi)$, the stationary points are still of the form $(y_i,y_i-y_j)$. It is easy to see that the eigenvalues of the Jacobian are now
\begin{align*}
\lambda_1(y;\Xi)&=-(2+\alpha)+\sum^r_{l=1}\Xi_l \frac{\partial}{\partial y}F_1\Bigl(y;z,\elr\Bigr)\\
&=-(2+\alpha)+\sum^r_{l=1}\Xi_l \lambda_1\Bigl(y;\elr\Bigr);\\
\lambda_2(y,d;\Xi)&=-(2+\alpha)+\sum^r_{l=1}\Xi_l \frac{\partial}{\partial d}F_2\Bigl(y,d;\elr\Bigr)\\
&=-(2+\alpha)+\sum^r_{l=1}\Xi_l \lambda_1\Bigl(y-d;\elr\Bigr).
\end{align*}
Consequently Theorem \ref{only one eig needed} can also be extended to this case: if we have $y_i,y_j\in\mathcal{Y}(z)$ with $y_i\geq y_j$ that satisfy $\lambda_1(y_i;\Xi)<0$ and $\lambda_1(y_j;\Xi)<0$ then $(y_i,y_i-y_j)$ is a stable equilibrium of the vector field given by $F_1(y;z,\Xi)$ and $F_2(y,d;\Xi)$. This completes the proof of Theorem \ref{main2}.

Finally, the proof of Theorem \ref{main3} is the same as for the case $\Xi=\ekr$.

\subsection{Proofs of Theorems \ref{highlander} and \ref{threshold}}

\begin{proof}[Proof of Theorem \ref{highlander}]
By differentiating $F_1\Bigl(y;x,\ekr\Bigr)$ we show that for every $x$ there are at most two values of $\Psi(x)$ which occur with positive probability, and where there are two such values that they occur in two disjoint intervals which do not depend on $x$. Thus a point of condensation must almost surely involve a jump between these regions. Since $\Psi(x)$ is increasing by definition, this can happen at most once.

We have
\begin{align*}
\frac{\partial}{\partial y}g\Bigl(y;\ekr\Bigr)&=\frac{\partial}{\partial y}\sum_{i=k}^{r}\binom{r}{i}y^{i}(1-y)^{r-i}\\
&=\sum_{i=k}^{r}i\binom{r}{i}y^{i-1}(1-y)^{r-i}-\sum_{i=k}^{r-1}(r-i)\binom{r}{i}y^{i}(1-y)^{r-i-1} \\
&=\sum_{i=k}^{r}r\binom{r-1}{i-1}y^{i-1}(1-y)^{r-i}-\sum_{i=k}^{r-1}r\binom{r-1}{i}y^{i}(1-y)^{r-i-1} \\
&=r\binom{r-1}{k-1}y^{k-1}(1-y)^{r-k},
\end{align*}
because all other terms cancel. So
\[\frac{\partial}{\partial y} F_1\Bigl(y;x,\ekr\Bigr)=r\binom{r-1}{k-1}y^{k-1}(1-y)^{r-k}-(2+\alpha);\]
note that this does not depend on $x$.

If $k=r$ then $\frac{\partial^{2}}{\partial y^2}F_1\Bigl(y;x,\ekr\Bigr)$ is positive on $(0,1)$, and if $k=1$ then it is negative on $(0,1)$. Otherwise,
\begin{align*}
\frac{\partial^2}{\partial y^2} F_1\Bigl(y;x,\ekr\Bigr)&=r\binom{r-1}{k-1}\bigl((k-1)(1-y)-(r-k)y\bigr)y^{k-2}(1-y)^{r-k-1}\\
&=r\binom{r-1}{k-1}y^{k-2}(1-y)^{r-k-1}\bigl((k-1)-(r-1)y\bigr),
\end{align*}
which is positive for $y\in\bigl(0,\frac{k-1}{r-1}\bigr)$ and negative for $y\in\bigl(\frac{k-1}{r-1},1\bigr)$. It follows that, for any choice of $k$, the equation
\begin{equation}r\binom{r-1}{k-1}y^{k-1}(1-y)^{r-k}-(2+\alpha)=0\label{turning}\end{equation}
has at most two roots in $(0,1)$, and that if it has exactly two such roots $z_1<z_2$ then the left-hand side is positive for $y\in(z_1,z_2)$.

Suppose \eqref{turning} has two roots $z_1<z_2$. Then for any $x$ we have that $F_1\Bigl(y;x,\ekr\Bigr)$ is strictly decreasing on the intervals $[0,z_1)$ and $(z_2,1]$, but strictly increasing on $(z_1,z_2)$. Consequently, $F_1\Bigl(y;x,\ekr\Bigr)=0$ has at most one root $y_1(x)\in[0,z_1]$, at most one root $y_2(x)\in(z_1,z_2)$ (which, if it exists, is an unstable zero), and at most one root $y_3(x)\in[z_2,1]$. Further, for each $i$, $y_i(x)$ is continuous on the range of $x$ for which it exists. Note that $F_1\Bigl(0;0,\ekr\Bigr)=F_1\Bigl(1;1,\ekr\Bigr)=0$ and so $y_1(0)=0$ and $y_3(1)=1=\Psi(1)$. By Theorem \ref{main1}, almost surely for almost all $x$ we have $\Psi(x)\in\{y_1(x),y_3(x)\}$. Let $x^*=\inf\{x\in[0,1]:\Psi(x)=y_3(x)\}$; since $\Psi$ is increasing and $y_1(x)$ and $y_3(x)$ are continuous we have $\Psi(x)=y_1(x)$ for all $x\in[0,x^*)$ and $\Psi(x)=y_3(x)$ for all $x\in(x^*,1]$. It follows that $\Psi$ is continuous everywhere except at $x^*$, as required.

If \eqref{turning} has exactly one root, $z$, in $(0,1)$, then the equation $F_1\Bigl(y;x,\ekr\Bigr)=0$ has at most one root $y_1(x)\in[0,z]$ and at most one root $y_2(x)\in(z,1]$ for every $x\in [0,1]$; again we must have $y_1(0)=0$ and $y_2(1)=1=\Psi(1)$. Defining $x^*$ as above, almost surely $\Psi(x)=y_1(x)$ for all $x\in[0,x^*)$ and $\Psi(x)=y_2(x)$ for all $x\in(x^*,1]$, so again $\Psi$ is continuous except possibly at $x^*$. Finally, if \eqref{turning} has no roots in $[0,1]$, then the equation $F_1\Bigl(y;x,\ekr\Bigr)=0$ has exactly one root $y_1(x)\in[0,1]$ for every $x\in[0,1]$ and we must have $y_1(0)=0$ and $y_1(1)=1=\Psi(1)$. Thus we almost surely have $\Psi(x)\equiv y_1(x)$, and there are no points of condensation.
\end{proof}

\begin{proof}[Proof of Theorem \ref{threshold}]
First, suppose $k=r$ (the case $k=1$ is similar). Then 
\[\frac{\partial}{\partial y} F_1\Bigl(y;x,\ekr\Bigr)=ry^{r-1}-(2+\alpha),\]
and $ry^{r-1}-(2+\alpha)=0$ has one root in $(0,1)$ if $\alpha>r-2$ and none otherwise. The proof of Theorem \ref{highlander} shows that almost surely condensation does not occur in the latter case. In the former case, let the root be $z$. Then 
$F_1\Bigl(y;x,\ekr\Bigr)$ is strictly increasing on $(z,1)$ meaning that any root of $F_1\Bigl(y;x,\ekr\Bigr)=0$ in this region is an unstable zero.
Consequently Theorem \ref{main1} implies that almost surely $\Psi(x)\leq z$ for all $x<1$, i.e.\ condensation occurs at $1$.

Secondly, suppose $1<k<r$. Then the left-hand side of \eqref{turning} is $-(2+\alpha)<0$ for $y=0$ and $y=1$, and is strictly increasing on
$y\in\bigl(0,\frac{k-1}{r-1}\bigr)$ and decreasing for $y\in\bigl(\frac{k-1}{r-1},1\bigr)$. Thus if it is negative for $y=\frac{k-1}{r-1}$ then \eqref{turning}
has no roots in $(0,1)$, and, as before, condensation almost surely does not occur. If the left-hand side is positive for $y=\frac{k-1}{r-1}$
then \eqref{turning} has two roots $z_1<z_2$ in $(0,1)$, and any root of $F_1\Bigl(y;x,\ekr\Bigr)=0$ in $(z_1,z_2)$ is an unstable zero, 
meaning that almost surely $\Psi(x)\not\in(z_1,z_2)$ for any $x$, so condensation occurs. If the left-hand side vanishes for $y=\frac{k-1}{r-1}$
then $F_1\Bigl(y;x,\ekr\Bigr)$ is strictly decreasing on $[0,1]$ and so $F_1\Bigl(y;x,\ekr\Bigr)=0$ has a single continuously-varying root $y_1(x)$ satisfying
$y_1(0)=0$ and $y_1(1)=1$; almost surely $\Psi(x)\equiv y_1(x)$ and so there is no condensation.

Since the sign of the left-hand side of \eqref{turning} changes precisely at $\alpha_{\mathrm c}$, this completes the proof.
\end{proof}
\section{Examples}\label{Section 4}
In this section we consider some examples of choices of $\Xi$ which illustrate how the results of Theorems \ref{main1} and \ref{main2} can apply to different cases.

\subsection{Middle of three}\label{midofthreemodd}

As $r=1$ gives standard preferential attachment, and the cases with $r=2$ and either $\Xi=(0,1)$ or $\Xi=(1,0)$ are equivalent to cases of the choice-fitness model of \cite{Sou.1}, which has rather different behaviour, the simplest case which illustrates our results is the ``middle of three'' model given by $r=3$ and $\Xi=(0,1,0)$. Here Theorem \ref{threshold}
predicts a phase transition at $\alpha_{\mathrm c}=-1/2$.

We can express our functions $F_1\Bigl(y;x,\ekr\Bigr)$ and $F_2\Bigl(y,d;\ekr\Bigr)$ as \[F_1\Bigl(y;x,\midthree\Bigr)=-2y^3+3y^2-(2+\alpha)y+x(\alpha+1)\] and \[F_2\Bigl(y,d;\midthree\Bigr)=-2d^{3}+6d^{2}y-3d^{2}-6dy^{2}+6dy-d(2+\alpha).\]

For $x\in(0,1)$, define $\{\psi_1(x),\psi_2(x),\psi_3(x)\}$ such that $\psi_1(x)\leq\psi_2(x)\leq\psi_3(x)$ as the three real roots of $F_1\Bigl(y;x,\midthree\Bigr)=0$ when three exist and $\psi(x)$ as the single root when only one exists. We have $F_1'\Bigl(y;x,\midthree\Bigr)=-6y^2+6y-(2+\alpha)$, and so $F_1\Bigl(y;x,\midthree\Bigr)$ is decreasing in $y$ whenever $\alpha\geq-\frac12$, but has turning points at $y=\frac12\pm\sqrt{-(1+2\alpha)/12}$ for $\alpha\in\bigl(-1,-\frac12\bigr)$. At $x=\frac12$ we have $F_1\Bigl(y;x,\midthree\Bigr)=\bigl(\frac12-y\bigr)\Bigl(F_1'\Bigl(y;x,\midthree\Bigr)+1+2\alpha\Bigr)/3$. Consequently the values of $F_1\Bigl(y;\frac12,\midthree\Bigr)$ at the turning points are $\mp\sqrt{-(1+2\alpha)^3/108}$, and the corresponding values of $F_1\Bigl(y;\frac12,\midthree\Bigr)$ are given by $(1+\alpha)\bigl(x-\frac12\bigr)\mp\sqrt{-(1+2\alpha)^3/108}$. It follows that there are multiple roots if and only if $\bigl| x-\frac12\bigr|\leq\sqrt{\frac{-(1+2\alpha)^3}{108(1+\alpha)}}$; write $s=s(\alpha)$ for this quantity. Note that $s<\frac12$ if and only if $\alpha>-\frac78$. 

\begin{figure}[!htb]
\begin{tikzpicture} 
\begin{axis}[
width=.95\textwidth,
height=6cm,
axis lines = left,
ymin=-0.15,ymax=0.2,
xlabel = $y$,
ylabel = {$F_1\Bigl(y;x,\midthree\Bigr)$},
]
\addplot [
domain=0:1, 
samples=100, 
color=black,
dotted,
thick,
]
{-2*x^3+3*x^2-1.25*x+0.159};
\addlegendentry{$x=\frac{9+\sqrt{6}}{18}$}
\addplot [
domain=0:1, 
samples=100, 
color=black,
dashed,
thick,
]
{-2*x^3+3*x^2-1.25*x+0.091};
\addlegendentry{$x=\frac{9-\sqrt{6}}{18}$}
\addplot [
domain=0:1, 
samples=100, 
color=black,
]
{0};
\end{axis}
\end{tikzpicture}
\caption{$F_1\Bigl(y;x,\midthree\Bigr)$ for $\alpha=-0.75$ and $x=\frac12\pm s$.}\label{Fig_2}
\end{figure}
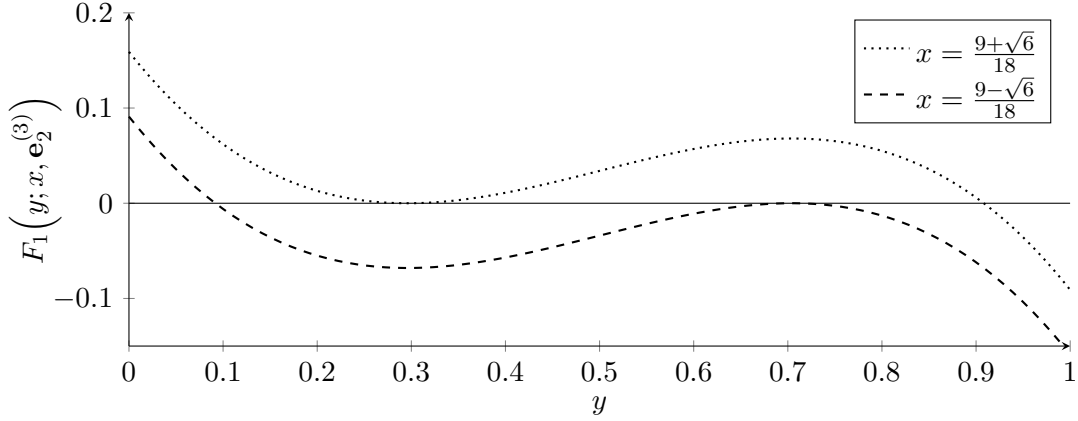

Figure \ref{Fig_2} plots $F_1\Bigl(y;x,\midthree\Bigr)$ against $y\in[0,1]$ for the value $\alpha=-\frac{3}{4}$ and $x=\frac12\pm s$, and Figure \ref{Fig_4} plots the roots of $F_1\Bigl(y;x,\midthree\Bigr)$ against $x\in[0,1]$. There is exactly one real root when $x\in\Bigl\{\Bigl[0,\frac{9+\sqrt{6}}{18}\Bigr)\cup\Bigl(\frac{9-\sqrt{6}}{18},1\Bigr]\Bigr\}$ and three real roots when $x\in\Bigl[\frac{9-\sqrt{6}}{18},\frac{9+\sqrt{6}}{18}\Bigr]$.

\begin{figure}[!htb]
\centering
\includegraphics[width=\textwidth]{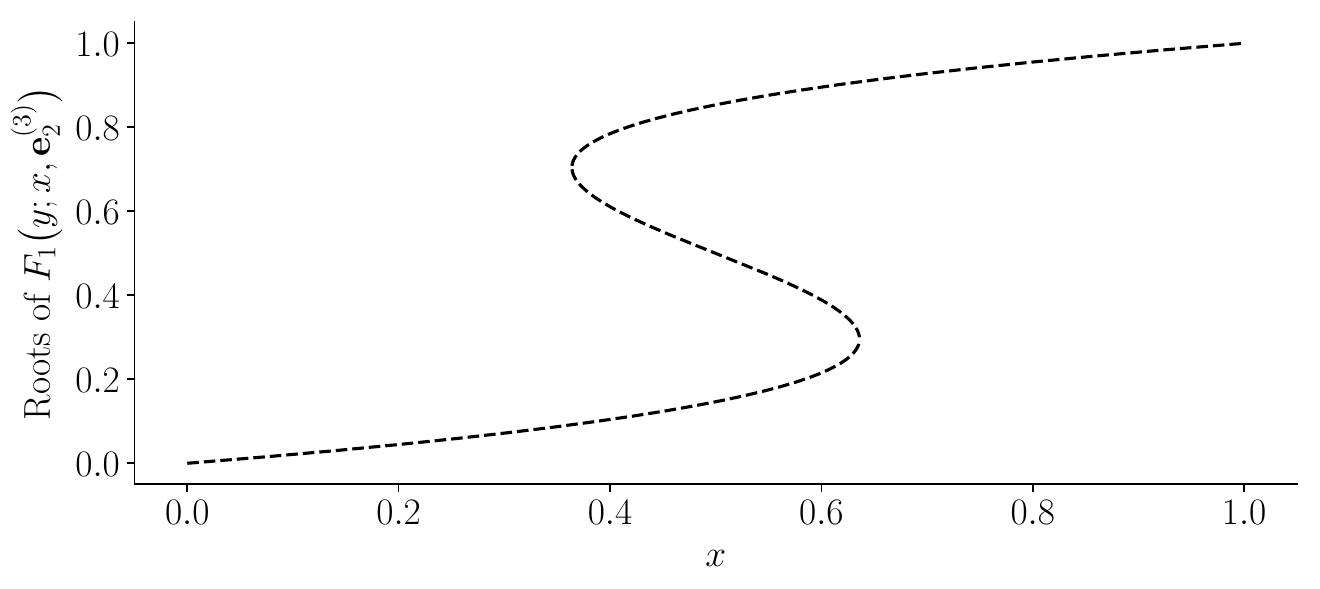}
\caption{The roots of $F_1\Bigl(y;x,\midthree\Bigr)$ for $x\in[0,1]$ and $\alpha=-0.75$.}\label{Fig_4}
\end{figure}

In this setting, Theorem \ref{main1} becomes the following.
\begin{thm}\label{mid of 3 phase trans}
For a fixed location $x\in(0,1)$, the random variable $\Psi_n(x)$ converges pointwise as $n\to\infty$ almost surely to the following limits.
\begin{equation*}
\lim_{n\to\infty}\Psi_n(x)=
\begin{cases}
\psi(x) ,& \text{if } \alpha\geq-\frac{1}{2}\\
\psi(x) ,& \text{if } \alpha\in\bigl(-\frac{7}{8},-\frac{1}{2}\bigr)\text{ and }x\not\in\bigl[\frac{1}{2}-s,\frac{1}{2}+s\bigr]\\
\psi_{1}(x)\text{ or }\psi_{3}(x),& \text{if } \alpha\in\bigl(-\frac{7}{8},-\frac{1}{2}\bigr)\text{ and } x\in\bigl[\frac{1}{2}-s,\frac{1}{2}+s\bigr]\\
\psi_{1}(x)\text{ or }\psi_{3}(x),& \text{if } \alpha\leq-\frac{7}{8}.
\end{cases}
\end{equation*}
\end{thm}
We can see there is a phase transition at $\alpha=-\frac{1}{2}$:
when $\alpha\geq-\frac{1}{2}$, $\Psi$ is almost surely continuous, whereas when $\alpha<-\frac{1}{2}$, $\Psi$ follows the lower root $\psi_1(x)$
until a random point in $\bigl[\frac{1}{2}-s,\frac{1}{2}+s\bigr]$ at which it jumps to the upper root $\psi_3(x)$, giving a point of condensation. 

\begin{figure}[!htb]
\begin{minipage}[b]{.5\textwidth}\vspace{0pt}
\centering\includegraphics[width=\textwidth]{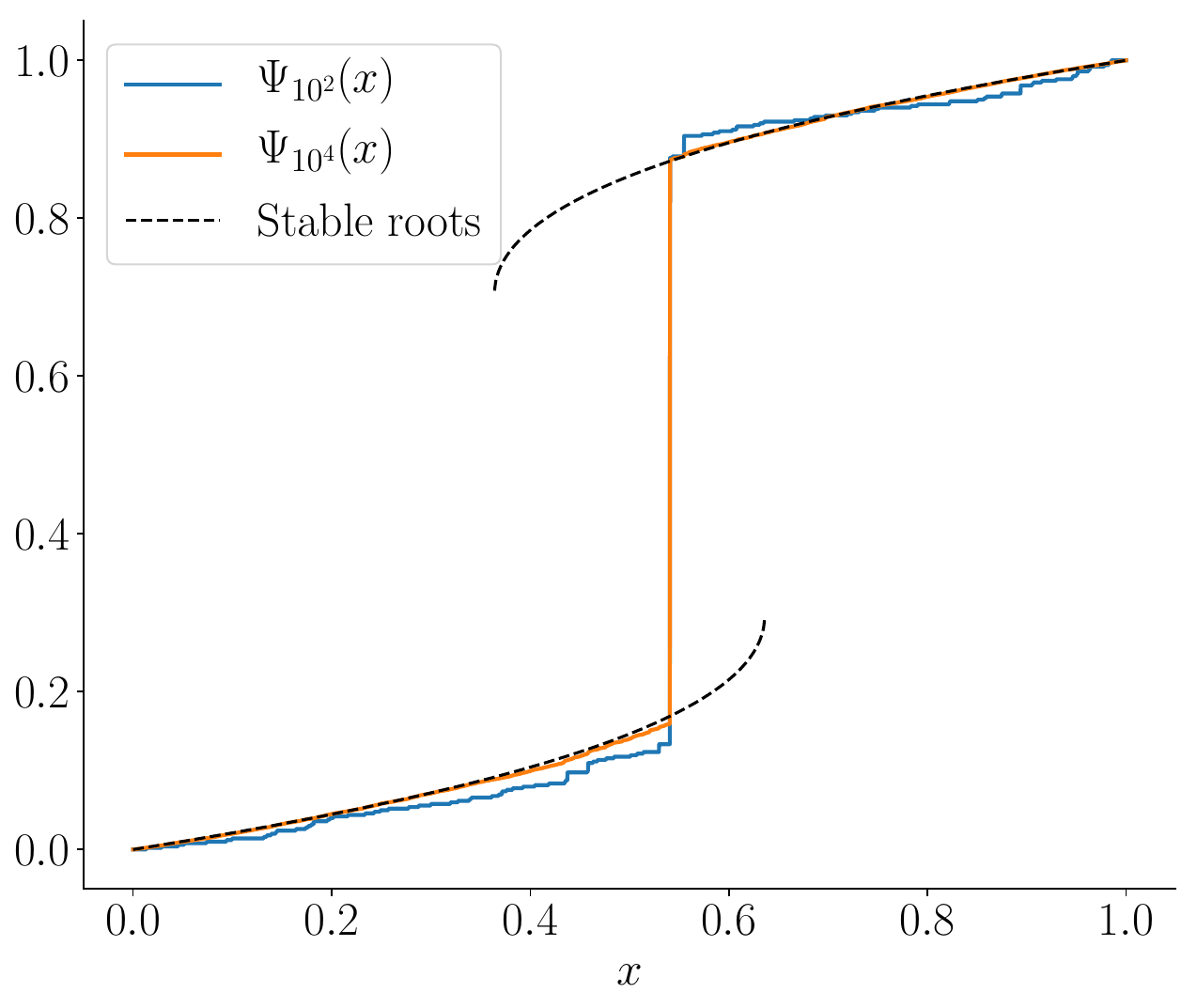}
\end{minipage}%
\begin{minipage}[b]{.5\textwidth}\vspace{0pt}
\centering\includegraphics[width=\textwidth]{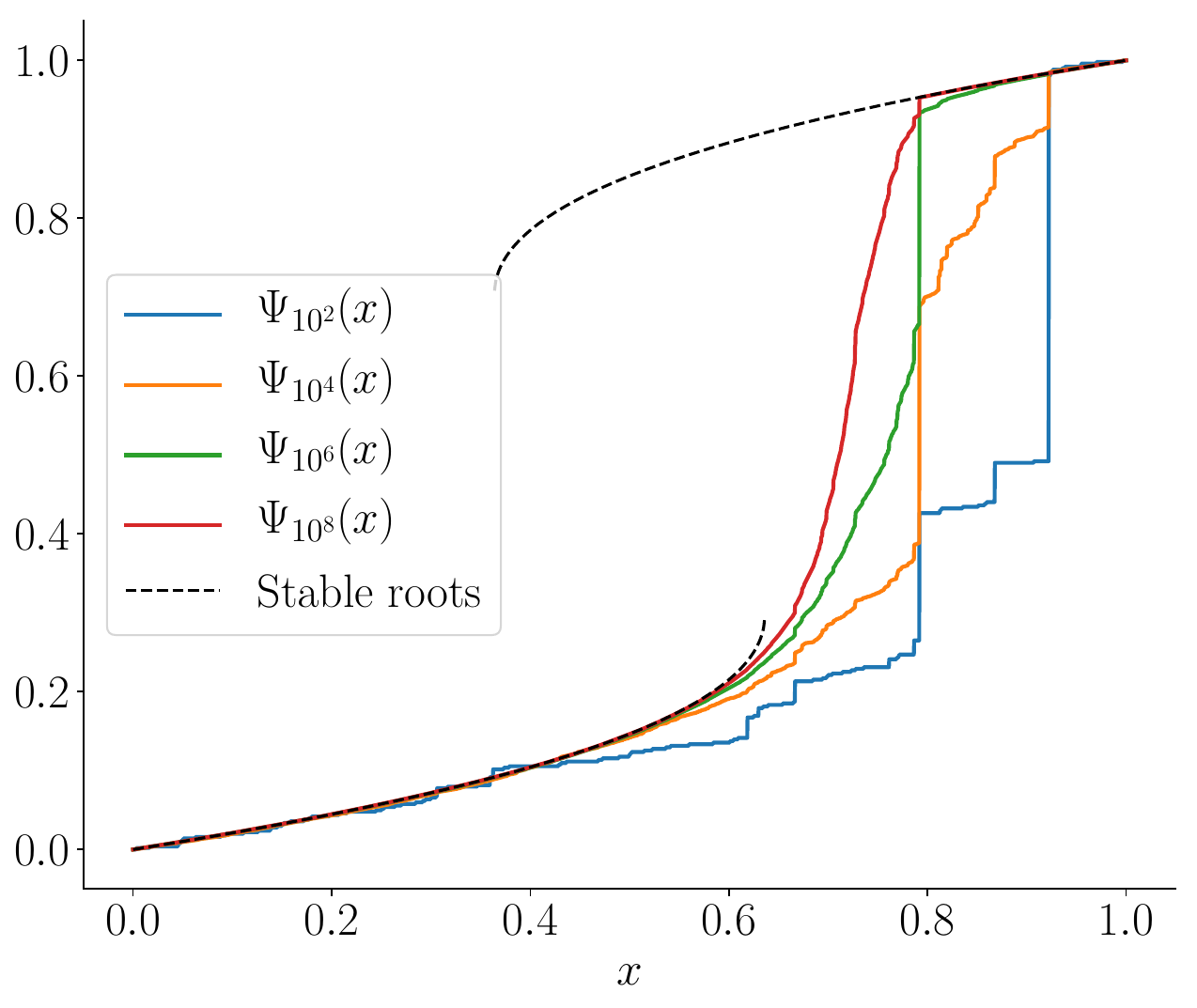}
\end{minipage}
\caption{Results from simulations for $\alpha=-0.75$}\label{mid-sims}
\end{figure}

If $-\frac{7}{8}<\alpha< -\frac{1}{2}$, Theorem \ref{main2} implies that this point of condensation is with positive probability caused by a persistent hub 
occurring at a random location which has full support on $\bigl(\frac{1}{2}-s,\frac{1}{2}+s\bigr)$. However, Theorem \ref{main3} implies that the point of condensation
also has positive probability of occurring at each of the endpoints $\frac{1}{2}-s$ and $\frac{1}{2}+s$; since almost surely these values are not locations 
of any vertex, it follows that there is also a positive probability that there is no persistent hub. Figure \ref{mid-sims} shows the results of two simulations
for $\alpha=-0.75$ with different behaviour: in the first simulation there is rapid convergence of $\Psi_n$ to a limit with condensation occurring via a 
persistent hub, whereas in the second $\Psi_n$ shows much slower convergence, apparently towards condensation at $\frac12+s$.
If $\alpha\leq-\frac{7}{8}$, Theorem \ref{main2} implies that the location of the jump has full support on $(0,1)$.

As we can now implement conditions on $F_1\Bigl(y,x;\ekr\Bigr)$ using $x$ and $\alpha$ to control whether we have one or three real roots we can solve
\begin{equation*}
F_1\Bigl(y,x;\ekr\Bigr)=F_2\Bigl(y,d,x;\ekr\Bigr)=0
\end{equation*}
by assuming $F_1\Bigl(y,x;\ekr\Bigr)=0$ has three real roots $\{\psi_1(x),\psi_2(x),\psi_3(x)\}$. We therefore can solve $F_2\Bigl(y,d,x;\ekr\Bigr)=0$ to get $\delta_1=0$ and $\delta_2$ and $\delta_3$ given by $\frac{3}{4}(2\psi_i(x)-1)\pm\frac{1}{4}\sqrt{-12\psi_i(x)^2+12\psi_i(x)-7-8\alpha}$.

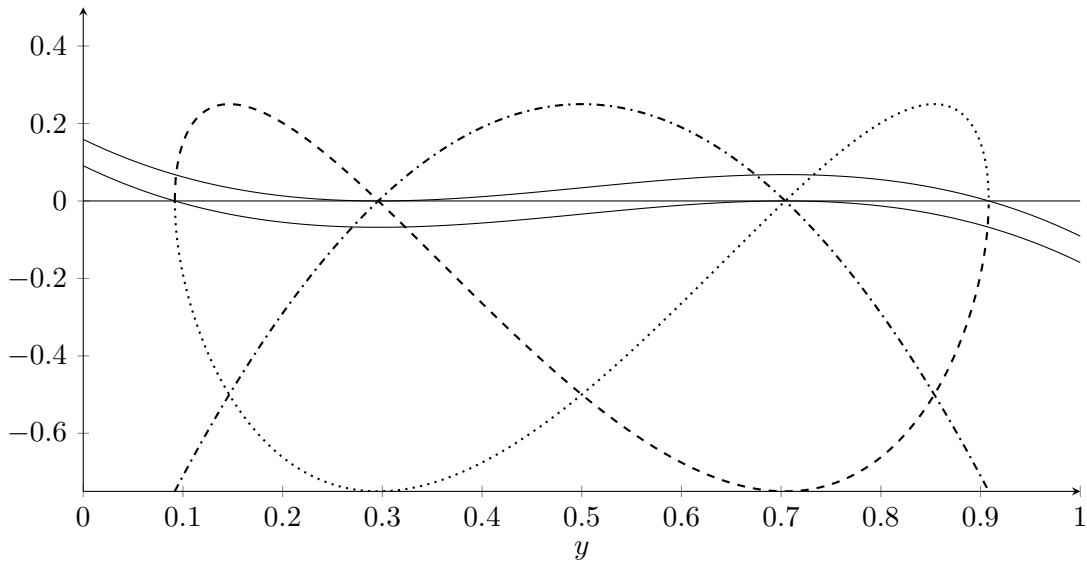
\begin{figure}[!htb]
\begin{tikzpicture}
\begin{axis}[
width=\textwidth,
height=8cm,
axis lines = left,
ymin=-0.75,ymax=0.5,
xlabel = $y$,
xmin=0, xmax=1,
]
\addplot [
domain=0.0917517:0.908248, 
samples=300, 
color=black,
thick,
dash dot,
]
{-6*x^2+6*x-1.25};
\addplot [
domain=0.0917517:0.908248, 
samples=1000, 
color=black,
dashed,
thick,
]
{-0.25*((-12*x^2 + 12*x -7 +6)^0.5)*(6*x-3+(-12*x^2 + 12*x -7 +6)^0.5)};
\addplot [
domain=0.0917517:0.908248, 
samples=1000, 
color=black,
dotted,
thick,
]
{-0.25*((-12*x^2 + 12*x -7 +6)^0.5)*(3-6*x+(-12*x^2 + 12*x -7 +6)^0.5)};
\addplot [
domain=0:1, 
samples=300, 
color=black,
]
{-2*x^3 + 3*x^2 - 1.25*x + 0.25*(0.5-(6^0.5)/18)};
\addplot [
domain=0:1, 
samples=300, 
color=black,
]
{-2*x^3 + 3*x^2 - 1.25*x + 0.25*(0.5+(6^0.5)/18)};
\addplot [
domain=0:1, 
samples=100, 
color=black,
]
{0};
\end{axis}
\end{tikzpicture}
\caption{Plot of eigenvalues of the Jacobian when $\alpha=-0.75$}\label{Fig_5}
\end{figure}

Figure \ref{Fig_5} illustrates Theorem \ref{only one eig needed} in this case, showing the eigenvalues of the Jacobian at the stationary points.
The solid curves show $F_1\Bigl(y;x,\midthree\Bigr)$ for $y\in[0,1]$ at $x=\frac{1}{2}\pm s$, the upper and lower limits of the region of $x$ where there are three real roots. In this same region are plotted $\lambda_1\Bigl(y;\midthree\Bigr)$ (the parabola), $\lambda_2\Bigl(y,\delta_2;\midthree\Bigr)$ and $\lambda_2\Bigl(y,\delta_3;\midthree\Bigr)$ (dashed and dotted lines respectively). The two regions where the eigenvalues are both negative overlap with where the roots of $F_1\Bigl(y;x,\ekr\Bigr)$ would be as $x$ increases from the lower limit to the upper limit.

\subsection{Second or sixth of seven}\label{secondorsixth}
The second example we will discuss makes use of the vector notation introduced in Section \ref{Section 2}. The ``middle of three'' model of Section \ref{midofthreemodd} is an example of selecting the $k$th highest location from $r$ selections, and demonstrates a phase transition below which condensation must occur at a single point. By Theorem \ref{highlander}, no such model can have condensation occurring simultaneously at more than one point. We now consider whether models in which more than one rank has positive probability of being selected can demonstrate multiple points of condensation. If there are three (or more) stable roots of $F_1(y;x,\Xi)$ for some range of $x$ then by Theorem \ref{main2} there is a positive probability of a jump from the first to the second in that range, and in this case since there are still higher stable roots, another jump must occur. If there are two disjoint ranges with two or more stable roots, separated by a range in which there is only one, then at least one jump must occur in each of these ranges. In this section we give an example which (for different values of $\alpha$) demonstrates that both of these can occur, even for models where only two ranks have positive probability. A real-life example of when two points of condensation might be expected is that of a bipartisan election, where two candidates from different regions of the location parameter (which might represent political position) may both attract a given proportion of the votes. 

The distribution we shall use is $\Xi=\bigl(0,\frac{1}{2},0,0,0,\frac{1}{2},0\bigr)$, that is, each new vertex is equally likely to connect to the second or sixth rank of seven candidates; this was the simplest example we could find which allowed for two points of condensation. With the motivating example above, it would be reasonable to suppose that vertices whose own location is higher are more likely to choose the sixth rank of their candidate pool; recall that this type of dependence is permitted in the model.

In this setting, \eqref{prob select vert leq x} gives the following expression for $F_1(y;x,\Xi)$:
\begin{align*}
F_1(y;x,\Xi)&=\frac{1}{2}F_1\Bigl(y;x,\basisvector 27\Bigr)+\frac{1}{2}F_1\Bigl(y;x,\basisvector 67\Bigr)\\
&=\frac{1}{2}\Biggl(\sum^{7}_{i=2}\binom{7}{i}y^{i}(1-y)^{7-i}\Biggr)+\frac{1}{2}\Biggl(\sum^{7}_{i=6}\binom{7}{i}y^{i}(1-y)^{7-i}\Biggr)\\
&\phantom{=}-(2+\alpha)y+x(\alpha+1)\\
&=-6y^7+21y^6-42y^5+\frac{105}{2}y^4-35y^3+\frac{21}{2}y^2-(2+\alpha)y +x(\alpha+1).
\end{align*}

The middle of three model features two phase transitions, at $\alpha=-\frac{1}{2}$ and $\alpha=-\frac{7}{8}$. To discuss phase transitions in this model, we note that the derivative (with respect to $y$) $F_1'(y;x,\Xi)$ does not depend on $x$ and is decreasing in $\alpha$ for fixed $y$; we also note that the symmetry in the system means that $F_1'(y;x,\Xi)=F_1'(1-y;x,\Xi)$. We can thus define
\[\alpha_1=\inf\{\alpha: F_1'(y;x,\Xi)\leq 0\; \forall\; y\in(0,1)\};\]
for $\alpha\geq \alpha_1$, $F_1(y;x,\Xi)$ is a decreasing function of $y$ and so for all $x\in(0,1)$ there is a unique root of $F_1(y;x,\Xi)=0$
in $(0,1)$, whereas for $\alpha<\alpha_1$ there is at least one interval of values of $x$ which have at least three roots of $F_1(y;x,\Xi)=0$ in
$(0,1)$. Hence our results show that condensation occurs almost surely if and only if $\alpha<\alpha_1$. We can calculate $\alpha_1$ explicitly,
since
\[F_1''(y;x,\Xi)=-\tfrac72(2y-1)\Bigl(6y^2-6y+4-\sqrt{10}\Bigr)\Bigl(6y^2-6y+4+\sqrt{10}\Bigr)\]
does not depend on $\alpha$. It is easy to verify that $F_1'(y;x,\Xi)$ is maximised at $y=\frac12\pm\frac16\sqrt{6\sqrt{10}-15}$,
and the maximum value is positive if and only if $\alpha<\frac{35\sqrt{10}-116}{9}$.

For $\alpha\in(-1,\alpha_1)$ $F_1(y;x,\Xi)$ has, in $(0,1)$, two local minima at $\eta_1(\alpha)$ and $\eta_3(\alpha)$ and two local maxima at $\eta_2(\alpha)$ and $\eta_4(\alpha)$, where $\eta_1(\alpha)<\eta_2(\alpha)<\eta_3(\alpha)<\eta_4(\alpha)$; these values depend on $\alpha$ but not on $x$. Set $\alpha_2=\sup\{\alpha: F_1(\eta_2(\alpha);x,\Xi)\geq F_1(0;x,\Xi)\}$; then the set of values of $x$ which have at least three roots includes $0$ and $1$ if and only if $\alpha\leq\alpha_2$. Next, set $\alpha_3=\sup\{\alpha: F_1(\eta_4(\alpha);x,\Xi)\geq F_1(\eta_1(\alpha);x,\Xi)\}$; then for $\alpha<\alpha_3$ there is a range of values of $x$ such that there are five roots of $F_1(y;x,\Xi)=0$ in $(0,1)$, whereas for $\alpha>\alpha_3$ there are always at most three, and there are two disjoint intervals of $x$ where there are three. Hence for $\alpha\in (\alpha_3,\alpha_1)$ there will almost surely be two points of condensation, whereas for $\alpha<\alpha_3$ there will be positive probability of there being a single point of condensation.

Finally, set $\alpha_4=\sup\{\alpha: F_1(\eta_4(\alpha);x,\Xi)\geq F_1(0;x,\Xi)\}$; then there are five roots of $F_1(y;x,\Xi)=0$ in $(0,1)$ for all $x\in (0,1)$ if 
and only if $\alpha\leq \alpha_4$, and hence for this range of $\alpha$ a single point of condensation can occur at a location which is fully supported on $(0,1)$.

\begin{figure}[!htb]
\begin{subfigure}{.45\linewidth}
\begin{tikzpicture}[scale=1]
\begin{axis}[
width=\textwidth,
height=6cm,
axis lines = left,
ymin=-0.6,ymax=0.7,
xlabel = $y$,
xmin=0, xmax=1,
ymin=-0.15, ymax=0.15,
]
\addplot [
domain=0:1,
samples=300,
color=black,
dotted,
]
{-6*x^7+21*x^6-42*x^5+52.5*x^4-35*x^3+10.5*x^2-(2-0.59114)*x+0.5*(-0.59114+1)};
\addplot [
domain=0:1,
samples=300,
color=black,
dashed,
]
{-6*x^7+21*x^6-42*x^5+52.5*x^4-35*x^3+10.5*x^2-(2-0.59114)*x+1*(-0.59114+1)};
\addplot [
domain=0:1,
samples=300,
color=black,
]
{-6*x^7+21*x^6-42*x^5+52.5*x^4-35*x^3+10.5*x^2-(2-0.59114)*x+0*(-0.59114+1)};
\addplot [
domain=0:1,
samples=100,
color=black,
]
{0};
\end{axis}

\end{tikzpicture}
\caption{$\alpha=\alpha_1$} \label{fig:topleftgrid4}
\end{subfigure}%
\begin{subfigure}{.45\linewidth}
\centering
\begin{tikzpicture}[scale=1]
\begin{axis}[
width=\textwidth,
height=6cm,
axis lines = left,
ymin=-0.6,ymax=0.7,
xlabel = $y$,
xmin=0, xmax=1,
ymin=-0.15, ymax=0.15,
]
\addplot [
domain=0:1,
samples=300,
color=black,
dotted,
]
{-6*x^7+21*x^6-42*x^5+52.5*x^4-35*x^3+10.5*x^2-(2-0.87562)*x+0.5*(-0.87562+1)};
\addplot [
domain=0:1,
samples=300,
color=black,
dashed,
]
{-6*x^7+21*x^6-42*x^5+52.5*x^4-35*x^3+10.5*x^2-(2-0.87562)*x+1*(-0.87562+1)};
\addplot [
domain=0:1,
samples=300,
color=black,
]
{-6*x^7+21*x^6-42*x^5+52.5*x^4-35*x^3+10.5*x^2-(2-0.87562)*x+0*(-0.87562+1)};
\addplot [
domain=0:1,
samples=100,
color=black,
]
{0};
\end{axis}

\end{tikzpicture}
\caption{$\alpha=\alpha_2$} \label{fig:toprightgrid4}
\end{subfigure}
\\
\begin{subfigure}{.45\linewidth}
\centering
\begin{tikzpicture}[scale=1]

\begin{axis}[
width=\textwidth,
height=6cm,
axis lines = left,
ymin=-0.6,ymax=0.7,
xlabel = $y$,
xmin=0, xmax=1,
ymin=-0.15, ymax=0.15,
]
\addplot [
domain=0:1,
samples=300,
color=black,
dotted,
]
{-6*x^7+21*x^6-42*x^5+52.5*x^4-35*x^3+10.5*x^2-(2-0.93144)*x+0.5*(-0.93144+1)};
\addplot [
domain=0:1,
samples=300,
color=black,
dashed,
]
{-6*x^7+21*x^6-42*x^5+52.5*x^4-35*x^3+10.5*x^2-(2-0.93144)*x+1*(-0.93144+1)};
\addplot [
domain=0:1,
samples=300,
color=black,
]
{-6*x^7+21*x^6-42*x^5+52.5*x^4-35*x^3+10.5*x^2-(2-0.93144)*x+0*(-0.93144+1)};
\addplot [
domain=0:1,
samples=100,
color=black,
]
{0};
\end{axis}
\end{tikzpicture}
\caption{$\alpha=\alpha_3$} \label{fig:bottomleftgrid4}
\end{subfigure}%
\begin{subfigure}{.45\linewidth}
\centering
\begin{tikzpicture}[scale=1]
\begin{axis}[
width=\textwidth,
height=6cm,
axis lines = left,
ymin=-0.6,ymax=0.7,
xlabel = $y$,
xmin=0, xmax=1,
ymin=-0.15, ymax=0.15,
]
\addplot [
domain=0:1,
samples=300,
color=black,
dotted,
]
{-6*x^7+21*x^6-42*x^5+52.5*x^4-35*x^3+10.5*x^2-(2-0.96842)*x+0.5*(-0.96842+1)};
\addplot [
domain=0:1,
samples=300,
color=black,
dashed,
]
{-6*x^7+21*x^6-42*x^5+52.5*x^4-35*x^3+10.5*x^2-(2-0.96842)*x+1*(-0.96842+1)};
\addplot [
domain=0:1,
samples=300,
color=black,
]
{-6*x^7+21*x^6-42*x^5+52.5*x^4-35*x^3+10.5*x^2-(2-0.96842)*x+0*(-0.96842+1)};
\addplot [
domain=0:1,
samples=100,
color=black,
]
{0};
\end{axis}
\end{tikzpicture}
\caption{$\alpha=\alpha_4$} \label{fig:bottomrightgrid4}
\end{subfigure}
\caption{$F_1(y;x,\Xi)$ evaluated at four different values of $\alpha$ corresponding to the phase transitions that appear for this choice of $\Xi$, and at three different values of $x$: from top to bottom, $x=1$, $x=\frac{1}{2}$ and $x=0$.}\label{Fig_9}
\end{figure}
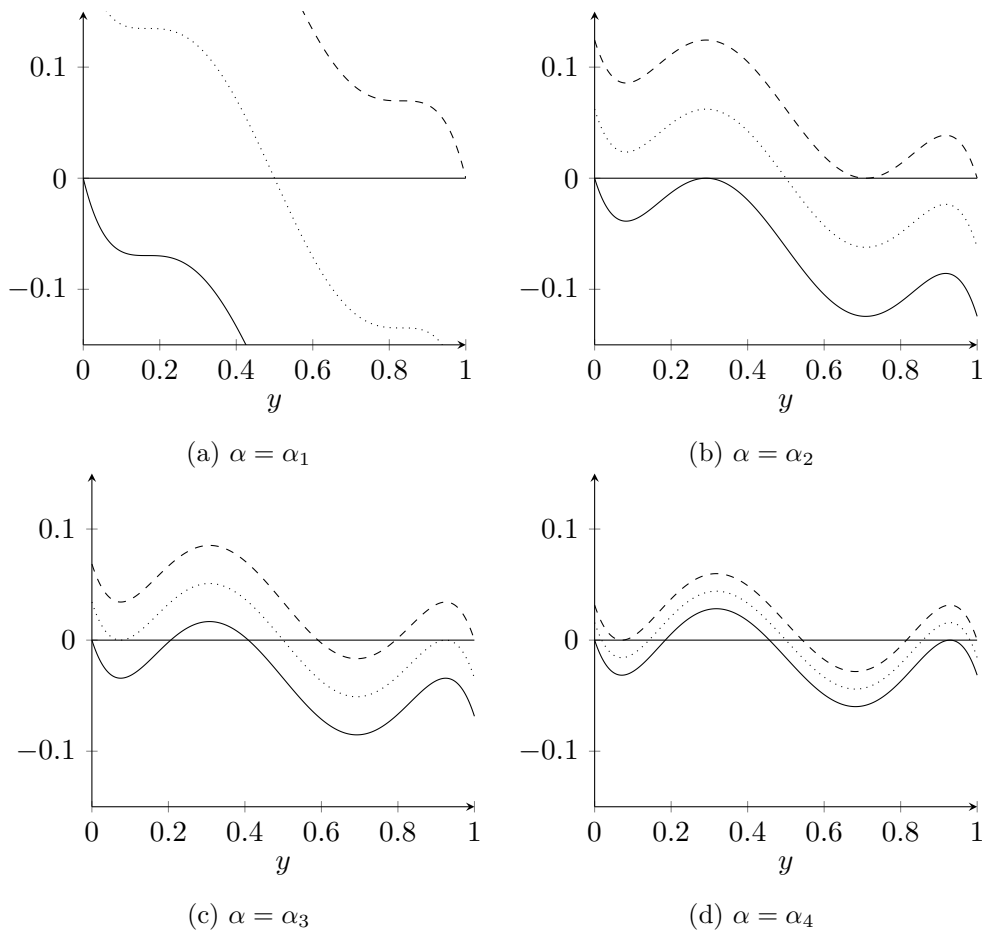 

The four transition points satisfy $\alpha_{1}=\frac{35\sqrt{10}-116}{9}\approx-0.59114$,
$\alpha_{2}\approx-0.87562$, $\alpha_{3}\approx-0.93144$ and $\alpha_{4}\approx-0.96842$. Plots of $F_1(y;x,\Xi)$ for each of the transition points $\alpha_1$, $\alpha_2$, $\alpha_3$ and $\alpha_4$ are shown in Figure \ref{Fig_9}, and plots showing the roots of $F_1(y;x,\Xi)=0$ for two specific values of $\alpha$ ($\alpha=-0.85\in (\alpha_2,\alpha_1)$ and $\alpha=-0.95\in (\alpha_4,\alpha_3)$) appear in Figure \ref{sec-six-roots}.

\begin{figure}[!htb]
\begin{minipage}[b]{.5\textwidth}\vspace{0pt}
\centering\includegraphics[width=\textwidth]{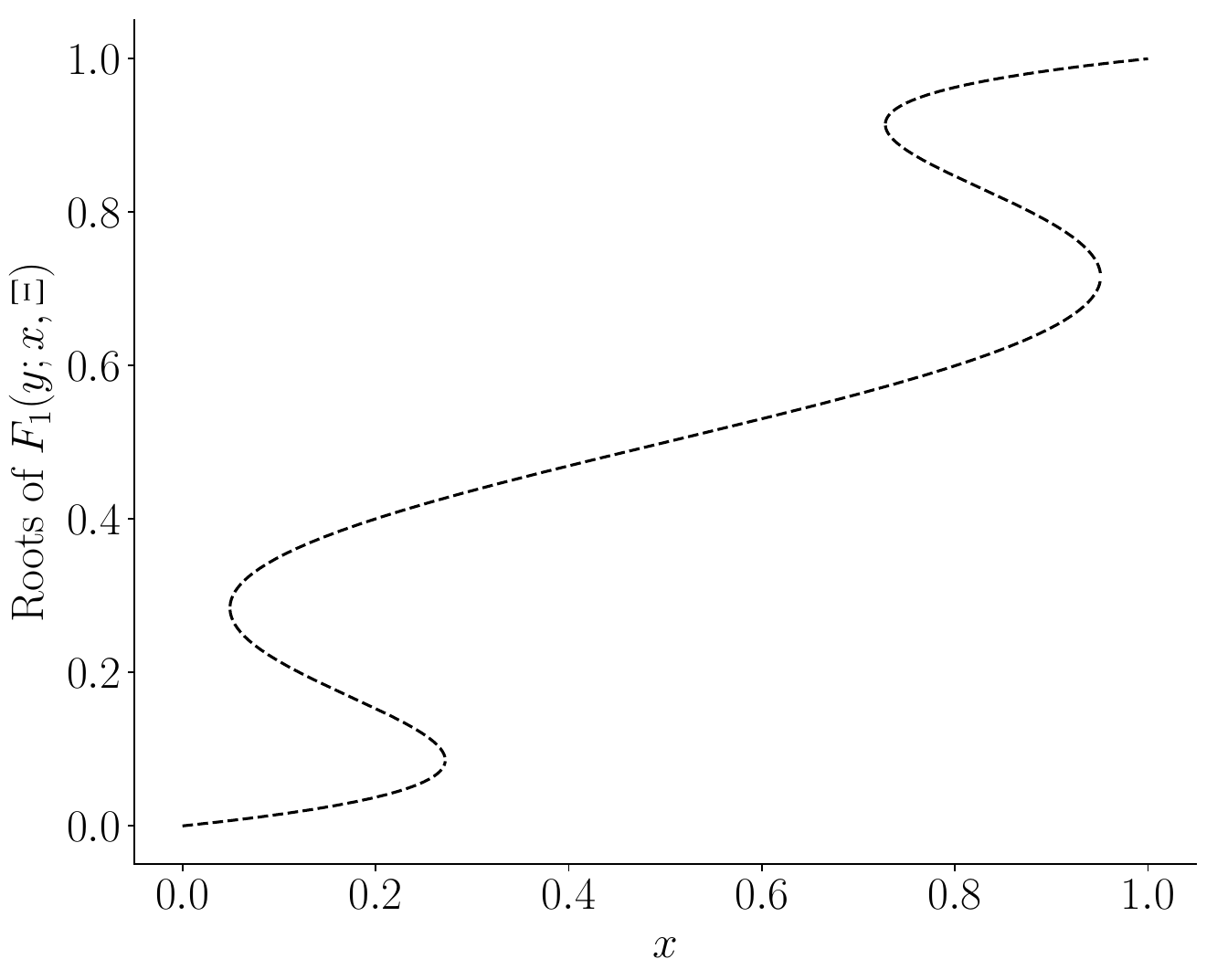}
\end{minipage}%
\begin{minipage}[b]{.5\textwidth}\vspace{0pt}
\centering\includegraphics[width=\textwidth]{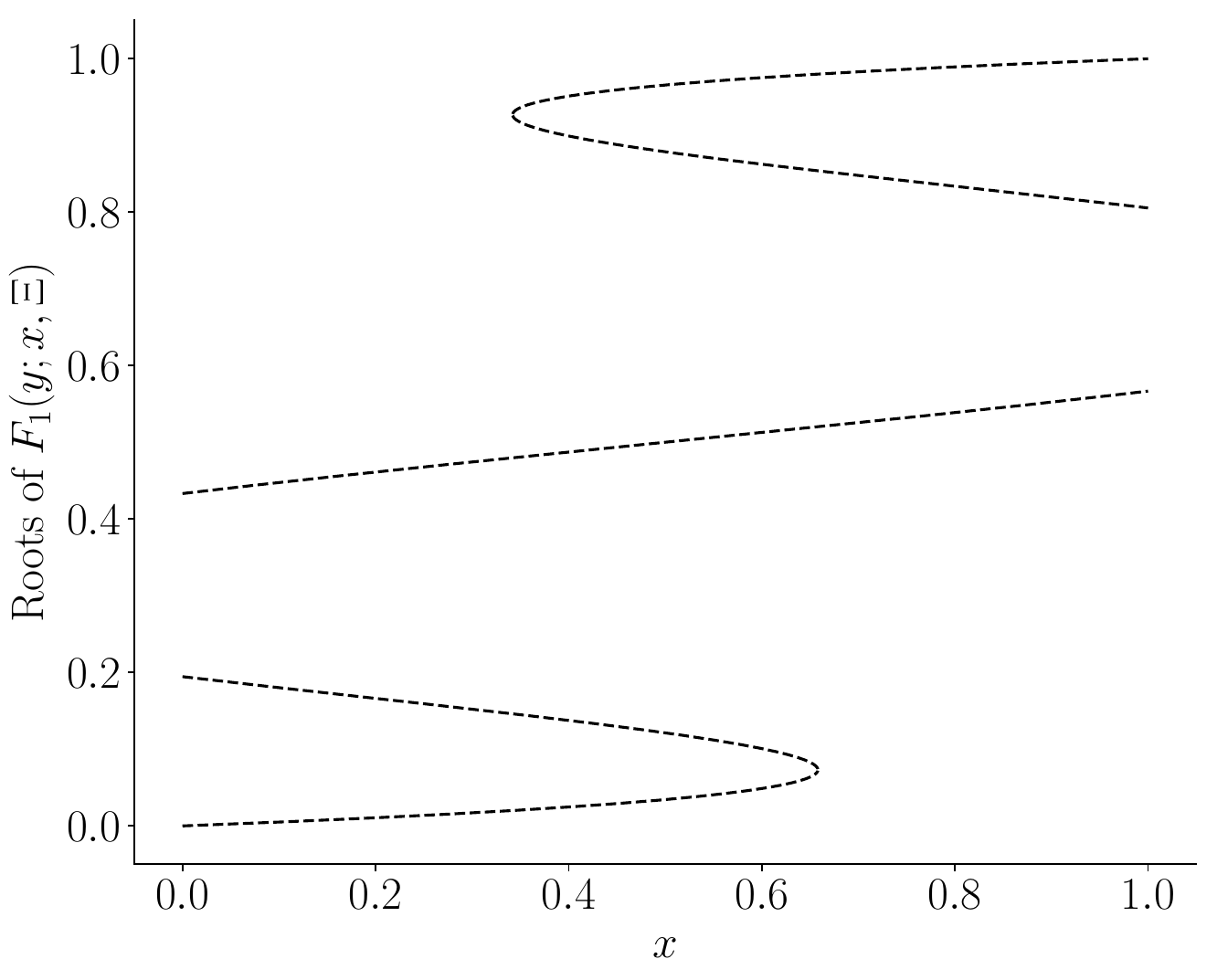}
\end{minipage}
\caption{The roots of $F_1(y;x,\Xi)$ for $\alpha=-0.85$ (left) and $\alpha=-0.95$ (right).}\label{sec-six-roots}
\end{figure}

For $\alpha=-0.85$ the stable roots and touchpoints are given by three continuous partial functions of $x$: $\psi_1$ defined on $(0,\beta_2]$,
$\psi_3$ defined on $[\beta_1,1-\beta_1]$, and $\psi_5$ defined on $[1-\beta_2,1)$, where $0<\beta_1<\beta_2<1/2$. (In fact we have $\beta_1\approx 0.0492$ and $\beta_2\approx 0.2721$.) Each function gives a stable root on the interior of its domain and a touchpoint on the boundary. Consequently, by Theorem \ref{main1}, almost surely $\Psi(x)$ takes the value $\psi_1(x)$ on some interval containing $(0,\beta_1)$, the value $\psi_3(x)$ on some interval containing $(\beta_2,1-\beta_2)$, and the value $\psi_5(x)$ on some interval containing $(1-\beta_1,1)$. Thus there are almost surely two points of condensation. By Theorem \ref{main2}, each of these points of condensation is caused by a persistent hub with positive probability. However, Theorem \ref{main3} implies there is also a positive probability of condensation occurring without a persistent hub at $\beta_1$, $\beta_2$, $1-\beta_2$ or $1-\beta_1$. Figure \ref{85sims} shows the results of two simulations: in the first simulation, both points of condensation arise from persistent local hubs, but in the second the upper part of $\Psi_n$ shows much slower convergence, apparently towards condensation at $1-\beta_2$. Our results do not give any bounds on the relative probabilities of these types of behaviour; however, simulations suggest that it is relatively uncommon to have early hubs forming in both the feasible regions $(\beta_1,\beta_2)$ and $(1-\beta_2,1-\beta_1)$.

\begin{figure}[!htb]
\begin{minipage}[b]{.5\textwidth}\vspace{0pt}
\centering\includegraphics[width=\textwidth]{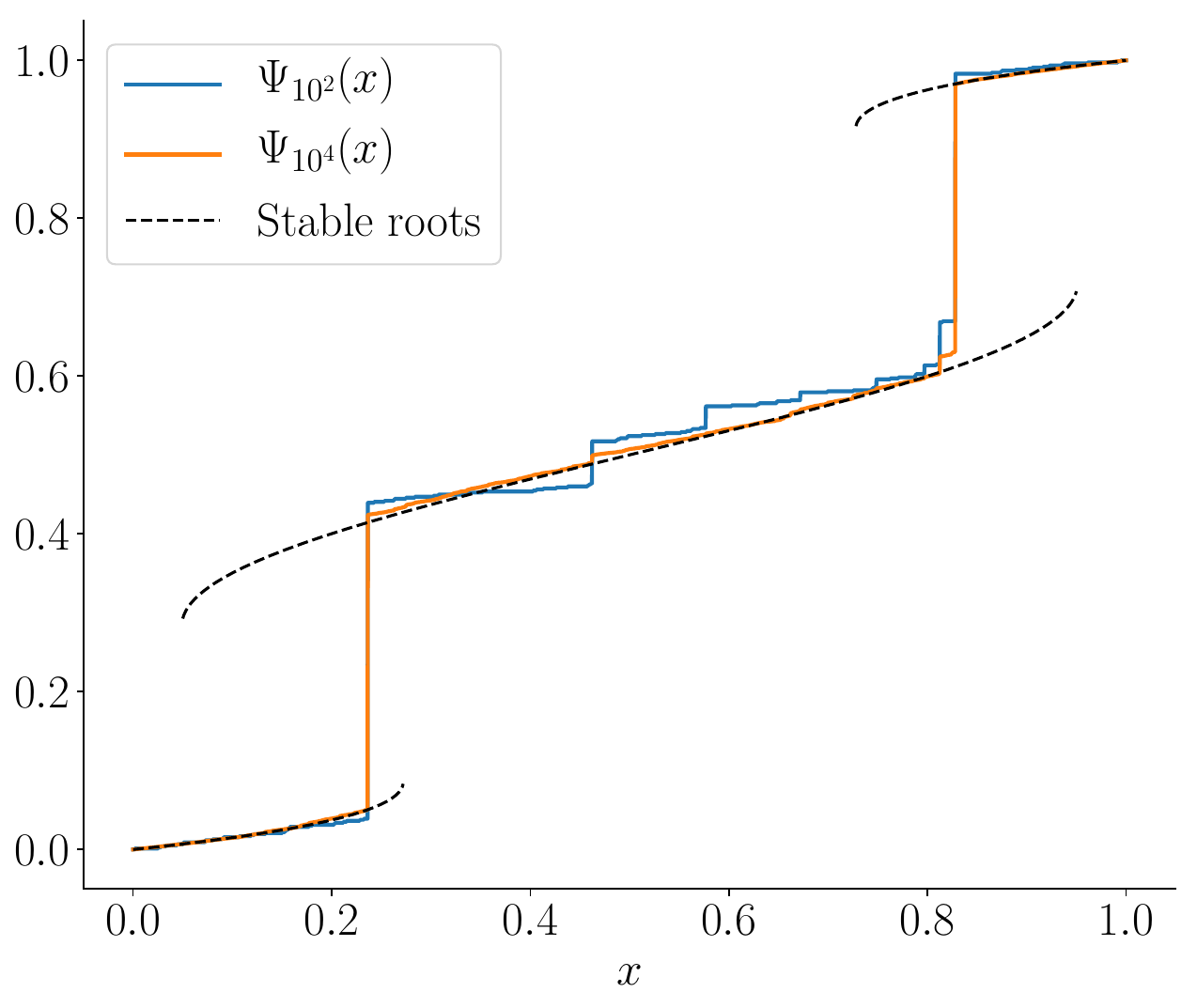}
\end{minipage}%
\begin{minipage}[b]{.5\textwidth}\vspace{0pt}
\centering\includegraphics[width=\textwidth]{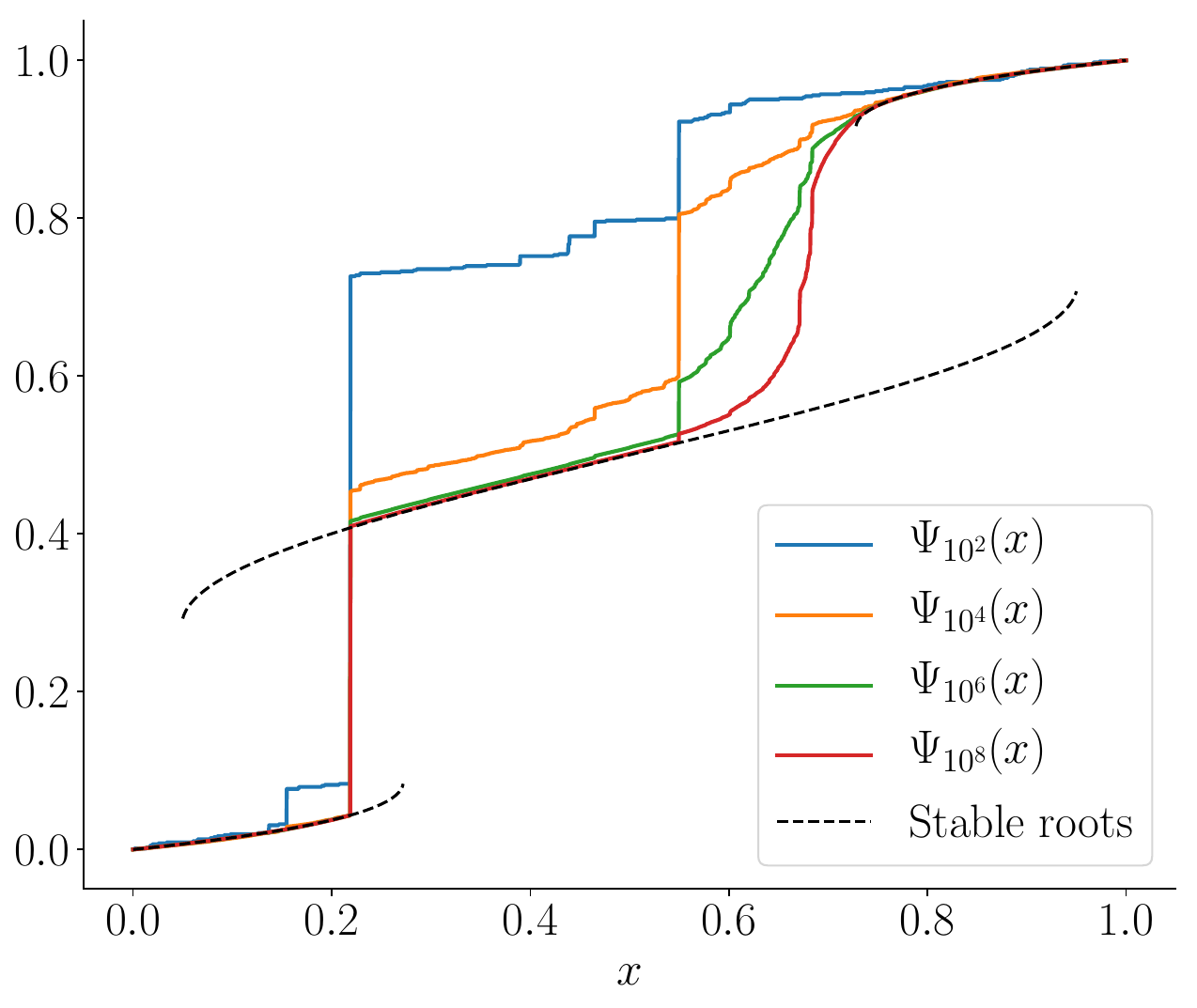}
\end{minipage}
\caption{Results from simulations for $\alpha=-0.85$}\label{85sims}
\end{figure}

For $\alpha=-0.95$ the corresponding partial functions of $x$ are $\psi_1$ defined on $(0,1-\beta]$,
$\psi_3$ defined on $(0,1)$, and $\psi_5$ defined on $[1-\beta,1)$, where $\beta\approx0.3420$. In this case Theorem \ref{main1} implies that there are nontrivial regions on which $\Psi(x)$ takes the values $\psi_1(x)$ and $\psi_5(x)$, but there need not be any $x$ for which $\Psi(x)=\psi_3(x)$, since that is never the only stable root. Consequently there may be two points of condensation if such an $x$ exists, and one point of condensation corresponding to a jump from $\psi_1$ to $\psi_5$ otherwise. Theorem \ref{main2} implies that each of these types of behaviour has positive probability, and Figure \ref{95sims} shows the results of two simulations exhibiting the two types of behaviour. As before, Theorem \ref{main3} implies that there is a positive probability of non-persistent condensation occuring at $\beta$ or $1-\beta$.

\begin{figure}[!htb]
\begin{minipage}[b]{.5\textwidth}\vspace{0pt}
\centering\includegraphics[width=\textwidth]{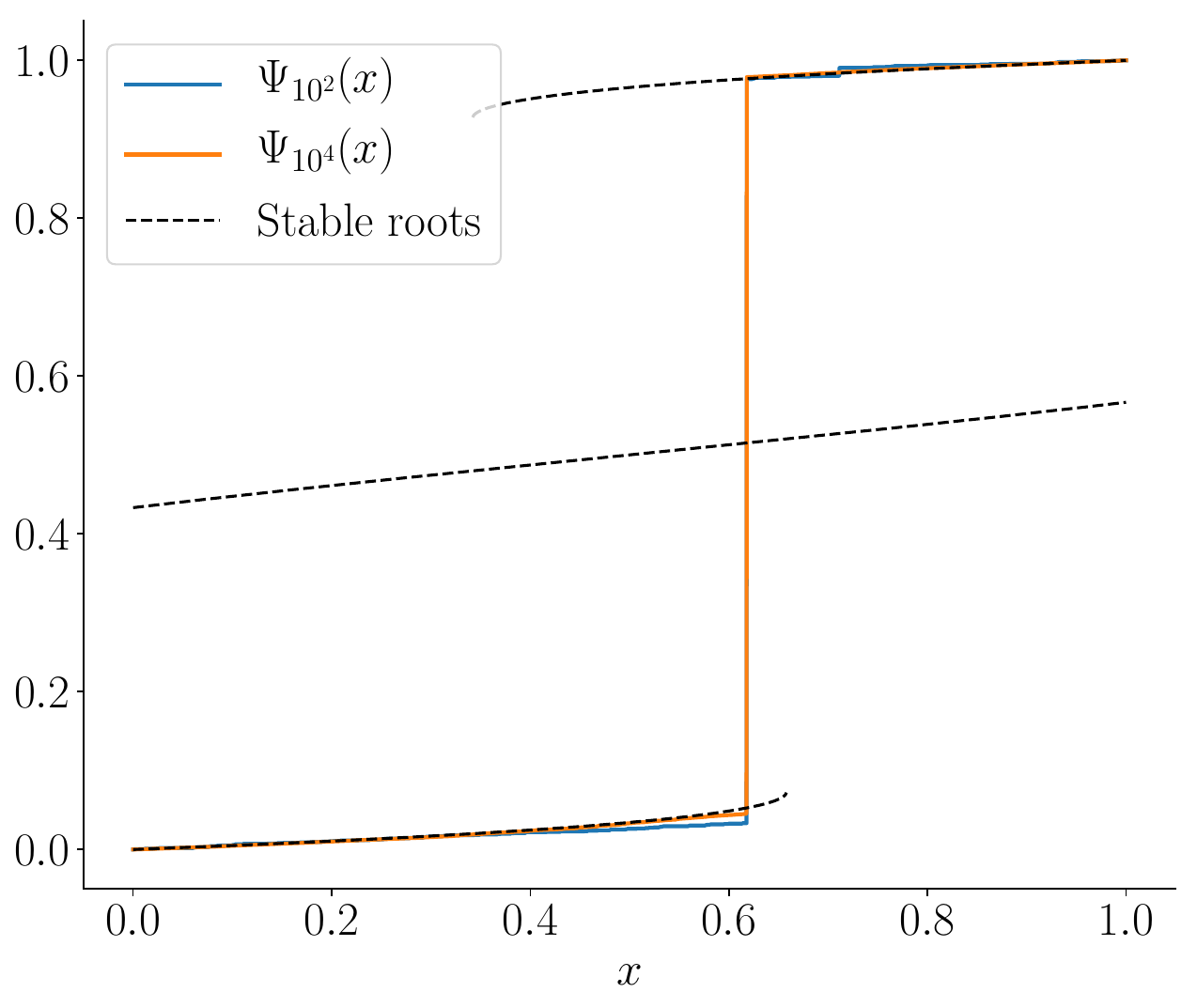}
\end{minipage}%
\begin{minipage}[b]{.5\textwidth}\vspace{0pt}
\centering\includegraphics[width=\textwidth]{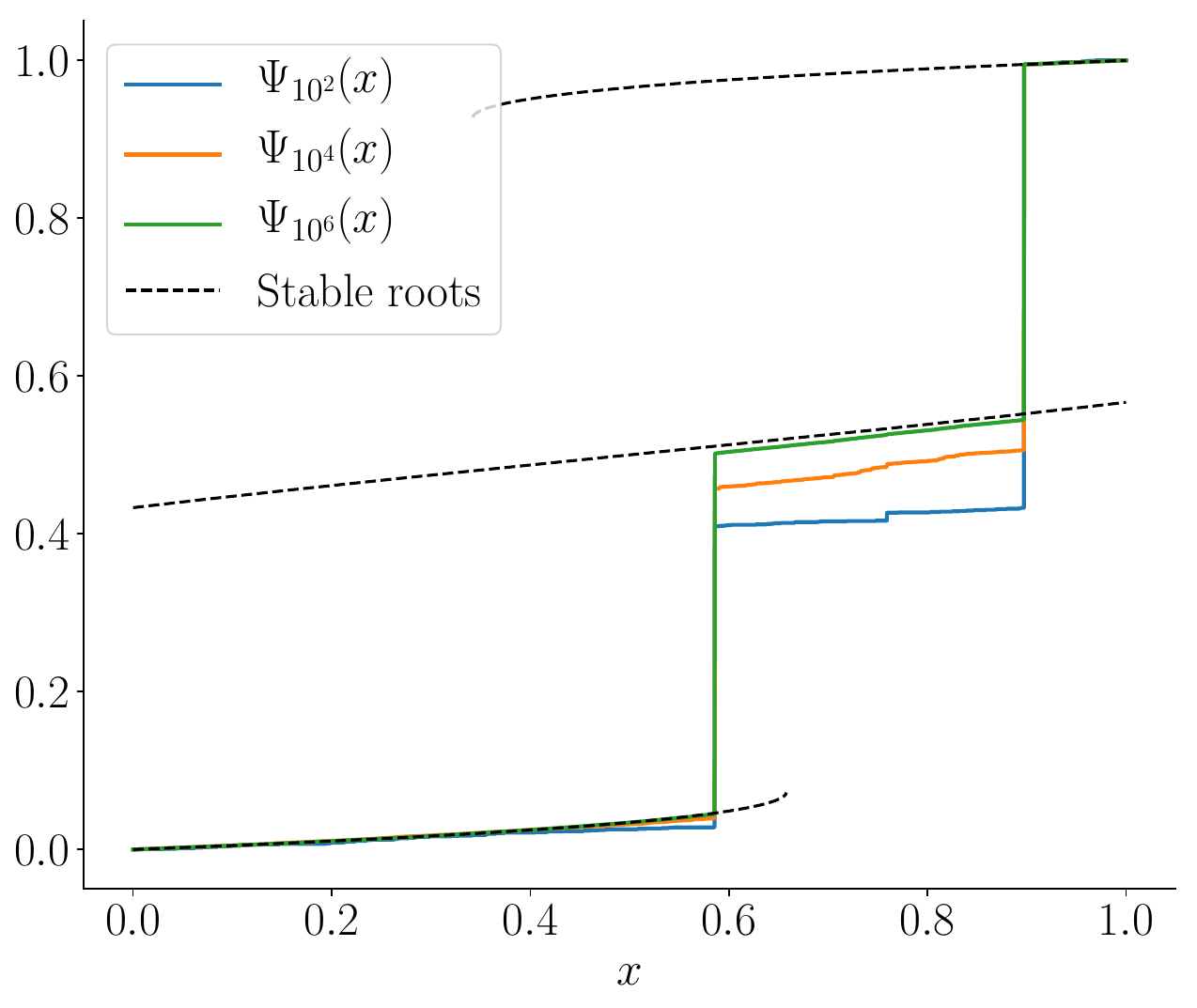}
\end{minipage}
\caption{Results from simulations for $\alpha=-0.95$}\label{95sims}
\end{figure}

\section{Acknowledgments}
The third author acknowledges support from the University of Sheffield for funding their research in the form of a graduate teaching assistant studentship, and would also like to thank Dr Nic Freeman for all of his feedback during the writing of this paper.

The first author acknowledges support from the European Research Council (ERC) under the European Union's Horizon 2020 research and innovation programme (grant agreement No 639046).

\bibliographystyle{abbrv}
\bibliography{references}
\end{document}